\documentclass{article}

\usepackage{natbib}
\usepackage{fullpage}

\usepackage[utf8]{inputenc} 
\usepackage[T1]{fontenc}    
\usepackage{hyperref}       
\usepackage{url}            
\usepackage{booktabs}       
\usepackage{amsfonts}       
\usepackage{nicefrac}       
\usepackage{microtype}      
\usepackage{xcolor}         
\usepackage{graphicx}
\usepackage{dsfont}
\usepackage{vwcol}
\usepackage{changepage}

\usepackage{mdframed}

\usepackage{shortcuts_OPT,enumitem,url,amsmath,amssymb,amsthm}

\newlength\figH
\newlength\figW
\usepackage[utf8]{inputenc}
\usepackage{pgfplots}
\DeclareUnicodeCharacter{2212}{−}
\usepgfplotslibrary{groupplots,dateplot}
\usetikzlibrary{patterns,shapes.arrows}
\pgfplotsset{compat=newest}

\usepackage{psfrag,subfigure,float,hyperref,bbm,cleveref}
\hypersetup{
  colorlinks   = true, 
  urlcolor     = blue, 
  linkcolor    = blue, 
  citecolor    = blue   
}

\makeatletter
\def\Exa@space@setup{%
  \Exa@preskip=5cm plus 2cm minus 2cm
  \Exa@postskip=\Exa@preskip 
}
\makeatother

\newtheoremstyle{exampstyle}
  {1.5\topsep} 
  {1.5\topsep} 
  {} 
  {} 
  {\bfseries} 
  {.} 
  {.5em} 
  {} 

\theoremstyle{exampstyle}
\newtheorem{Exa}{Example}

\newtheoremstyle{otherstyle}
  {1.5\topsep} 
  {1.5\topsep} 
  {\itshape} 
  {} 
  {\bfseries} 
  {.} 
  {.5em} 
  {} 

\theoremstyle{otherstyle}

\newtheorem{Lemma}{Lemma}

\newtheorem{Theorem}{Theorem}

\newtheorem{Corollary}{Corollary}

\newtheorem{assumption}{Assumption}

\newtheorem{manualtheoreminner}{Assumption}
\newenvironment{manualtheorem}[1]{%
  \renewcommand\themanualtheoreminner{#1}%
  \manualtheoreminner
}{\endmanualtheoreminner}

\title{State Dependent Performative Prediction with Stochastic Approximation}
\author{Qiang Li, Hoi-To Wai\thanks{Q.~Li and H.-T.~Wai are with the Department of Systems Engineering and Engineering Management, The Chinese University of Hong Kong, Hong Kong SAR of China. Emails: \texttt{\{liqiang, htwai\}@se.cuhk.edu.hk}}}
\date{\today}
\begin{document}

\maketitle

\begin{abstract}
	This paper studies the performative prediction problem which optimizes a stochastic loss function with data distribution that depends on the decision variable. We consider a setting where the agent(s) provides samples adapted to the learner's and agent's previous states. The said samples are used by the learner to optimize a loss function. This closed loop algorithm is studied as a state-dependent stochastic approximation (SA) algorithm, where we show that it finds a fixed point known as the \emph{performative stable} solution. Our setting models the unforgetful nature and the reliance on past experiences of agent(s). Our contributions are three-fold. First, we demonstrate that the SA algorithm can be modeled with \emph{biased} stochastic gradients driven by a controlled Markov chain (MC) whose transition probability is adapted to the learner's state. Second, we present a novel finite-time performance analysis of the state-dependent SA algorithm. We show that the expected squared distance to the performative stable solution decreases as ${\cal O}(1/k)$, where $k$ is the iteration number. Third, numerical experiments are conducted to verify our findings.
\end{abstract}

\section{Introduction}
Many supervised learning algorithms are built around the assumption that learners can obtain samples from a static distribution \emph{independent} of the state of the learner and/or the agent who provides the sample. This assumption is reasonable for static tasks such as image classification. Oftentimes, it simplifies the design and analysis of algorithms such as stochastic gradient methods.

On the other hand, in certain applications agents can be \emph{performative} where the samples are drawn from a \emph{decision-dependent} distribution. This is relevant to the framework of strategic classification \citep{hardt2016strategic, cai2015optimum,kleinberg2020classifiers}. For instance, when training a classifier for loan applications, given the classifier published by the \emph{learner} (bank), as the best response strategy, the \emph{agent(s)} (loan applicants) may manipulate their profile prior to the submission, e.g., by spending more with credit cards, making unnatural transactions, etc., in order to increase their chance of successful application. The latter manipulation effectively shifts the data distribution and may affect the convergence properties, or even the stability of learning algorithms.

Earlier works  \citep{bartlett1992learning,quinonero2009dataset} studied the effects of exogenous changes with shifts in data distribution. 
Recently, \cite{perdomo2020performative} considered the convergence properties of learning algorithms when they are agnostic to the shifted distributions. Specifically, the \emph{learner} is interested in the following \emph{performative prediction} problem:
\beq \label{eq:performative} 
\min_{ \theta \in \RR^d } ~ V(\theta) = \EE_{ z \sim {\cal D}(\theta) } \big[ \ell( \theta; z ) \big],
\eeq
where 
$\ell( \theta; z )$ is the loss function given the sample $z \in {\sf Z}$. The loss function is strongly-convex with respect to (w.r.t.) the parameter $\theta \in \RR^d$, and the gradient map $\grd_\theta \ell( \theta; z )$ is Lipschitz continuous w.r.t.~$z$, $\theta$. In addition, the distribution ${\cal D}(\theta)$ on ${\sf Z}$ is parameterized by the decision vector $\theta$, which captures the distribution shift due to the learner's state. In other words, problem \eqref{eq:performative} finds a parameter $\theta$ which minimizes the expected loss that takes care of the decision-dependent distribution.

Despite the strong convexity of $\ell(\theta ;z)$, problem \eqref{eq:performative} is non-convex in general due to the coupling with $\theta$ in the data distribution ${\cal D}(\theta)$. As a remedy, \cite{perdomo2020performative} studied population-based algorithms that converge to a \emph{performative stable} point, $\theta_{PS}$, which is a fixed point to the system $\theta = \argmin_{ \theta' \in \RR^d } \EE_{ z \sim {\cal D}(\theta) } \big[ \ell( \theta'; z ) \big]$. 
Along the same line, \cite{mendler2020stochastic} analyzed stochastic algorithms which deploy minibatches of i.i.d.~samples from the shifted distribution at each iteration, \cite{izzo2021learn, miller2021outside} studied gradient estimation techniques and developed algorithms that converge to an optimal solution of \eqref{eq:performative} through introducing a gradient correction term (also see \citep{munro2020learning} which has considered a related setting), \cite{drusvyatskiy2020stochastic} studied the stability of proximal stochastic gradient algorithms (and their variants) in the performative prediction setting, \cite{brown2020performative} studied population-based algorithms where the state-dependent distribution is updated iteratively. 

This paper studies the convergence of stochastic algorithms where only one sample (or a minibatch of samples) is required at each iteration. Specifically, we consider a greedy deployment scheme similar to \cite{mendler2020stochastic} where the learner deploys the most recent model after each update round. Moreover, the {agent}(s) is modeled with a memory property such that the update of his or her state depends on the past state. The closed-loop algorithm can be studied as a \emph{state-dependent stochastic approximation (SA)} algorithm. In contrast to the setting analyzed in \cite{mendler2020stochastic,drusvyatskiy2020stochastic} where only the learner's state is incrementally updated and the agent draws i.i.d.~samples from the distribution shifted by the learner's state, the agent's state evolves according to a \emph{controlled Markov chain (MC)} whose stationary distribution is the shifted distribution. 

Our study is motivated by the \emph{stateful (or unforgetful) nature} of the agents who depend on past experiences when adapting to a shifted target data distribution. For example, a loan applicant may take months to build up his/her credit history to adapt to changes in the published classifier. Several questions naturally arise from such dynamical performative prediction problems: \emph{will the stochastic algorithm converge to a performative stable point similar to \cite{mendler2020stochastic, drusvyatskiy2020stochastic}? what is the sample complexity?} This paper addresses these questions as we make the following contributions: 
\begin{itemize}[leftmargin=6mm]
	\item We develop a \emph{fully state-dependent performative prediction} framework which extends the analysis in \citep{mendler2020stochastic, drusvyatskiy2020stochastic}. The proposed extension relies on a state-dependent stochastic approximation (SA) algorithm with noise originating from a controlled Markov chain [cf.~Algorithm 1]. 
	\item Our main result consists of a finite-time convergence analysis of the state-dependent SA algorithm under a setting which \emph{does not assume} the iterates to be bounded a-priori. Previous works either assumed the latter condition a-priori (e.g., \cite{benveniste2012adaptive}), or they require a compact constraint set (e.g., \cite{atchade2017perturbed}). Using a novel analysis, we show that the mean squared error between the SA iterates and the unique performative stable solution [cf.~\eqref{eq:ps}] converges at a rate of ${\cal O}(1/k)$ in expectation. Additionally, we discuss about the convergence to an approximate stationary point of \eqref{eq:performative} when the loss function $\ell(\theta;z)$ is not strongly convex (possibly non-convex). 
	\item We demonstrate the efficacy of the SA algorithm with several experiments. We show that it has a comparable performance as in \cite{mendler2020stochastic} which assumes an ideal setting with i.i.d.~samples taken from the shifted distribution.
\end{itemize}
The rest of this paper is organized as follows. \S\ref{sec:main} formally describes the performative prediction problem and a state-dependent SA algorithm for tackling the problem, \S\ref{sec:result} presents the main theoretical results for the convergence of the state dependent SA algorithm, \S\ref{sec:pf} gives an overview of the proof strategy, and \S\ref{sec:num} presents the numerical experiments.

\paragraph{Related Works} Analysis for state-dependent stochastic approximation (SA) algorithms with controlled MC, which extend over the classical SA \citep{robbins1951stochastic}, has been considered in a number of works.  \citet{benveniste2012adaptive,kushner2003stochastic} studied the asymptotic convergence of such algorithms, also see \citep{tadic2017asymptotic} which analyzed the case of biased SA. 

Recent works have analyzed the finite-time performance of state-dependent SA algorithms that are related to ours. \citet{atchade2017perturbed} considered a proximal SA algorithm where the proximal function has a compact domain; \citet{karimi2019non} analyzed the plain SA algorithm without projection but assumed that the updates are bounded; \citet{sun2018markov, doan2020finite} studied SA algorithms with a static MC not suitable for performative prediction.
Our analysis relaxes the restrictions of these prior works and focuses on convergence to a performative stable solution unique to \eqref{eq:performative} [cf.~see \eqref{eq:ps}].

Lastly, our analysis technique is related to the recent endeavors on obtaining finite time bounds for reinforcement learning (RL) algorithms. Notice that  \eqref{eq:performative} can in fact be regarded as a special case of policy optimization \citep{sutton2018reinforcement}. To this end, recent works \citep{wu2020finite,xu2020non,zhang2020provably}  studied the sample complexity of actor critic algorithms in finding a stationary point of an average reward function, where  controlled MCs are considered. In comparison to our analysis, the latter works have only studied the convergence to a stationary point of a simple cost function.

\paragraph{Notations} We denote ${\sf Z}$ as the state space of samples provided by the agent(s) and ${\cal Z}$ is a $\sigma$-algebra on ${\sf Z}$. A Markov transition kernel is a map given by $\MK : {\sf Z} \times {\cal Z} \rightarrow \RR_+$. At the state $z \in {\sf Z}$, the next state is drawn as $z' \sim \MK( z, \cdot )$. It holds for any measurable function $f : {\sf Z} \rightarrow \RR$ that $\EE[ f(z') | z ] = \int_{\sf Z} f(z') \MK( z, {\rm d}\!~z') =: \MK f(z)$. Unless otherwise specified, the operator $\grd$ takes the gradient of a function w.r.t.~the first argument for a multivariate function, e.g., $\grd \ell( \theta; z )$ denotes the gradient of $\ell(\theta; z)$ taken w.r.t.~$\theta$. For $x, y \in \RR^d$, we denote the inner product as $\pscal{x}{y} = x^\top y$.

\section{State-dependent Performative Prediction} \label{sec:main}
Due to the intractability of Problem~\eqref{eq:performative}, we are interested in evaluating the \emph{performative stable} (PS) solution:
\beq \label{eq:ps}
\theta_{PS} = \argmin_{ \theta \in \RR^d } ~\EE_{ z \sim {\cal D}(\theta_{PS})} [ \ell( \theta; z ) ]. 
\eeq
Note that the expectation is taken with respect to $z \sim {\cal D}( \theta_{PS})$. It is known that $\theta_{PS}$ is in general different from an optimal or stationary solution to \eqref{eq:performative}; see \citep[Example 2.2]{perdomo2020performative}. Instead, $\theta_{PS}$ is a fixed point solution to the procedure when the learner repeatedly update the parameter $\theta$ with the drifted data distribution provided by the agent.

We consider a \emph{state dependent} stochastic approximation (SA) algorithm motivated by the stateful nature of agents.
The latter is modeled using a controlled Markov chain. 
For any $\theta$, we define a Markov transition kernel $\MK_{\theta}$ which induces a Markov chain with a unique stationary distribution $\pi_{\theta}(\cdot)$ such that $\EE_{z \sim \pi_{\theta}(\cdot)} [ \grd \ell(\theta; z) ] = \EE_{z' \sim {\cal D}(\theta)} [ \grd \ell( \theta ; z' )]$.
The \emph{learner} and \emph{agent} interact through the following rules:

\begin{center}
\begin{minipage}[c]{0.55\linewidth} 
\begin{center}
    \fbox{\begin{minipage}{.975\linewidth} 
    \underline{\textbf{Algorithm 1: State-dependent SA}}\vspace{.1cm}
    
    \textbf{Input}: initialization $\theta_0$, step sizes $\{ \gamma_k \}_{k \geq 0}$.\vspace{.1cm}
    
    \textbf{For} $k=0,1,2, \ldots$
    \begin{adjustwidth}{1.5em}{0pt}
    \emph{Agent} draws $z_{k+1} = {\sf SAMPLE}( \theta_k, z_k )$.\vspace{.1cm}
    
    \emph{Learner} updates\vspace{-.2cm}
    \beq \label{eq:sa}
	\theta_{k+1}=\theta_k-\gamma_{k+1} \grd \ell(\theta_{k}; z_{k+1}),\vspace{-.2cm}
    \eeq
    and deploys $\theta_{k+1}$.
    \end{adjustwidth}
    \end{minipage}}
\end{center}
\end{minipage}~~~
\begin{minipage}[c]{0.375\linewidth} 
\begin{center}
    \fbox{\begin{minipage}{.975\linewidth}
    \underline{${\sf SAMPLE}(\theta , z )$}:\vspace{.1cm}
    \begin{adjustwidth}{.5em}{0pt}
    Draw the next sample as \vspace{-.1cm}
    \beq \label{eq:mkv_setting}
    z' \sim \MK_{ \theta } ( z, \cdot ),\vspace{-.1cm}
    \eeq
    where $\MK_{\theta} : {\sf Z} \times {\cal Z} \rightarrow \RR_+$ is a Markov transition kernel.\vspace{.1cm}
    
    \textbf{Output}: $z' \in {\sf Z}$.
    \end{adjustwidth}
    \end{minipage}
    }
\end{center}
\end{minipage}
\end{center}

We observe that \eqref{eq:sa} is a standard SA recursion based on $\grd \ell(\theta_{k}; z_{k+1})$, where the \emph{learner} deploys the most recent model $\theta_k$ and takes the sample $z_{k+1}$ directly from the agent. 
Specifically, we consider a \emph{state-dependent} setting where the sampling of $z_{k+1}$ can be affected by both the \emph{learner}'s and \emph{agent}'s states. Formally, the state-dependency is captured by modeling the samples sequence $\{z_k\}_{k \geq 1}$ as a controlled MC in \eqref{eq:mkv_setting}. 
Notice that the stationary distribution $\pi_{\theta}(\cdot)$ does not need to be the same as ${\cal D}(\theta)$ as long as the former yields an unbiased gradient. However, for simplicity, we assume $\pi_{\theta}(\cdot) \equiv {\cal D}(\theta)$ for any $\theta \in \RR^d$. 

The {greedy deployment} scheme studied by \cite{mendler2020stochastic} assumed that the agents draw $z_{k+1}  {\sim} {\cal D}( \theta_k )$ as independent samples. The latter implies that $\grd \ell( \theta_k; z_{k+1})$ is an unbiased estimator of $\EE_{z' \sim {\cal D}(\theta_k)} [ \grd \ell( \theta_k; z' )]$. As a significant departure, in Algorithm~1, the stochastic gradient $\grd \ell( \theta_k; z_{k+1})$ is a \emph{biased estimator} for $\EE_{z' \sim {\cal D}(\theta_k)} [ \grd \ell( \theta_k; z' )]$. 
For any $\theta \in \RR^d$, we observe
\beq 
\begin{split}
& \EE[ \grd \ell ( \theta; z_{k+1} ) \!~|\!~ z_k ] = \MK_{\theta_k} \grd \ell( \theta; z_k ) = \int_{\sf Z} \grd \ell( \theta; z ) \MK_{\theta_k} ( z_k, {\rm d} \!~ z ) .
\end{split}
\eeq
Since $\MK_{\theta_k} ( z_k, \cdot ) \neq \pi_{\theta_k}(\cdot)$, we have $\EE[ \grd \ell ( \theta; z_{k+1} ) \!~|\!~ z_k ] \neq \EE_{z' \sim {\cal D}(\theta_k)} [ \grd \ell( \theta_k; z' )]$.
Under the setting \eqref{eq:mkv_setting} with restricted access to the shifted data distribution ${\cal D}(\theta_k)$, one possibility to obtain an \emph{unbiased} gradient estimate 
is to hold $\theta_k$ as fixed and repeat the sampling process $z' = {\sf SAMPLE}(\theta_k , z)$ indefinitely. In this case, we have the unbiased estimate $\lim_{ n \rightarrow \infty } \MK_{\theta_k}^n {\grd \ell}( \theta ; z ) = \EE_{z' \sim {\cal D}(\theta_k)} [ \grd \ell(\theta; z') ]$ for any initial agent state $z \in {\sf Z}$.   
Instead of searching for an unbiased gradient estimator through repeated MC updates, the state-dependent SA algorithm uses the {instantaneous} samples from \eqref{eq:mkv_setting} that are \emph{co-evolving} with the learner's iterate in \eqref{eq:sa}. 



\subsection{Example of Controlled Markov Chain \eqref{eq:mkv_setting}} \label{sec:exmkv}
The model \eqref{eq:mkv_setting} captures the stateful and stochastic nature of the agent as the sample $z_{k+1}$ depends on the previous one $z_k$. Concretely, our study is motivated by the following application example which satisfies $\pi_{\theta}(\cdot) \equiv {\cal D}(\theta)$ (a case where $\pi_{\theta}(\cdot) \neq {\cal D}(\theta)$ will be discussed in Appendix~\ref{app:gau}):
    
\begin{Exa}[Strategic Classification with Adapted Best Response] \label{ex:br} We consider the problem of strategic classification \citep{hardt2016strategic} involving some \emph{agents} and a \emph{learner}. In a typical scenario, the agent provides the best-response (i.e., optimized) samples upon knowing the current learner's state $\theta_k$. Ideally, the sample $z_{k+1} \sim {\cal D}(\theta_k)$ shall be drawn as 
	\beq \label{eq:br}  
	z_{k+1} \in \argmax_{ z' \in {\sf Z} }~U( z'; \tilde{z}_{k+1} , \theta_k ),~~\tilde{z}_{k+1} \sim {\cal D}_0,
	\eeq
where ${\cal D}_0$ is the base distribution and $U( z'; z, \theta )$ is strongly concave in $z'$ for any $(z,\theta)$. The best response $\max_{ z' \in {\sf Z} }U( z'; z , \theta )$ perturbs the base sample $z$ in favor of the agent.



In practice, the exact maximization in \eqref{eq:br} is not obtained unless the agent(s) are given sufficient time to respond to the learner's state. Instead, we consider a setting where the agent(s) improve their responses via a gradient ascent dynamics evolving simultaneously with the learner. 

Concretely, consider a setting where the learning problem \eqref{eq:performative} utilizes data provided by $m$ agents. 
Let ${\cal D}_0$ be the empirical distribution of $m$ data points $\{ \bar{d}_1,..., \bar{d}_m \}$, where $\bar{d}_i \in {\sf Z}$ is the initial data held by agent $i$.
At iteration $k$, a subset of agents ${\cal I}_k \subset \{1,...,m\}$ (selected uniformly with $|{\cal I}_k| = pm$, $p \in (0,1]$) becomes aware of the learner's state and they search for the best response through a gradient descent update. Then, the learner selects uniformly an agent $i_k \in \{1,...,m\}$ and requests the data sample $z_{k+1}$ from him/her. 
In summary, the inexact best response dynamics follows:
\beq \label{eq:ibr}
\textbf{Step 1:}~~d_i^{k+1} = 
\begin{cases} 
	d_i^k + \alpha \grd U( d_i^k; \bar{d}_i, \theta_k ), & i \in {\cal I}_k, \\
	d_i^k, & i \notin {\cal I}_k,
\end{cases} 
\qquad \textbf{Step 2:}~~z_{k+1} = d_{i_k}^{k+1} ,
\eeq
where the gradient is taken w.r.t.~the first argument of $U( \cdot )$, $\alpha > 0$ is the agents' response rate (stepsize). 
For the initialization, we set $d_i^0 = \bar{d}_i$ for all $i=1,...,m$. 
The above dynamics highlights the stateful nature of the agents as they improve the responses based on their past experiences.

The best response dynamics executed by the agent(s) in \eqref{eq:ibr} leads naturally to a controlled MC \eqref{eq:mkv_setting}. Specifically, the MC's state is given by the tuple $\hat{z}_k = \{ d_1^k, ..., d_m^k , z_k \}$ and the application of the Markov kernel $\MK_{\theta_k}$ to $\hat{z}_k$ yields the inexact best response dynamics \eqref{eq:ibr}. Furthermore, the latter admits a stationary distribution where 
$\lim_{n \rightarrow \infty} \MK_{\theta_k}^n \grd \ell( \theta; \hat{z} ) = \EE_{ z \sim {\cal D}(\theta_k)} [ \grd \ell( \theta; z ) ]$.
We provide detailed properties about the controlled MC in Appendix~\ref{app:br}. \hfill $\square$
\end{Exa}

The analysis of our state-dependent SA algorithm  \eqref{eq:sa}, \eqref{eq:mkv_setting} entails unique challenges not found in the literature. First, the agent's states $\{ z_k \}_{k \geq 1}$ form a \emph{controlled MC} whose transition probabilities are changing according to the learner's states $\{ \theta_k \}_{k \geq 0}$. Second, a plain SA update is used in \eqref{eq:sa} which does not require a compact constraint set as in the prior works such as \citep{atchade2017perturbed}. In fact, $\theta_k$ can be unbounded when the step size is not carefully selected.\vspace{-.1cm}

\section{Main Results} \label{sec:result}
\vspace{-.1cm}
Let us define the following shorthand notations  
\beq
f( \theta_1 ; \theta_2 ) = \EE_{z \sim {\cal D}(\theta_2) } [ \ell( \theta_1; z) ],~~\grd f(\theta_1;\theta_2) = \EE_{ z \sim {\cal D}(\theta_2) } [ \grd \ell( \theta_1 ; z) ],
\eeq
where the first argument $\theta_1$ controls the loss function value and the second argument $\theta_2$ controls the distribution shift. 
Notice that $\grd f(\theta_{PS}; \theta_{PS}) = 0$. We consider the following assumptions. First, the learner's loss is strongly convex in $\theta$, and its gradient map is Lipschitz continuous in $(\theta,z)$, i.e.,
\begin{assumption} \label{ass:strong}
	For each $z \in {\sf Z}$, there exists $\mu>0$ such that 
	\beq
	\ell( \theta ; z ) \geq \ell( \theta'; z ) + \pscal{ \grd \ell( \theta'; z ) }{ \theta - \theta' } + (\mu/2) \| \theta' - \theta \|^2, \quad \forall~\theta, \theta' \in \RR^d.
	\eeq
\end{assumption}
\begin{assumption}\label{ass:gradient}
	There exists $L \geq 0$ such that 
	\beq
	\begin{split}
	& \| \grd \ell( \theta; z ) -  \grd \ell( \theta' ; z' ) \| \leq L \big\{ \| \theta - \theta' \| + \| z - z' \| \big\}, \quad \forall~\theta, \theta' \in \RR^d,~z,z' \in {\sf Z}.
	\end{split}
	\eeq
\end{assumption}
Notice that as a consequence, the expected objective function $f(\theta_1, \theta_2)$ and gradient $\grd f(\theta_1; \theta_2)$ are $\mu$-strongly convex in $\theta_1$, and $L$-Lipschitz in $\theta_1$, respectively. These are standard assumptions in the optimization literature. As indicated by \citep{drusvyatskiy2020stochastic}, these conditions are necessary for finding a performative stable solution in \eqref{eq:ps}.

Second, we have the following assumption on the oscillation of the stochastic gradient $\grd \ell(\theta;z)$:
\begin{assumption}\label{ass:bounded} There exists $\sigma \geq 0$ such that 
	\beq \textstyle\textstyle	
	\sup_{ z \in {\sf Z} } \| \grd \ell( \theta; z ) - \grd f( \theta; \theta_{PS} ) \| \leq \sigma \big( 1 + \| \theta - \theta_{PS} \| ), \quad \forall~\theta \in \RR^d.
	\eeq
\end{assumption}
The above is slightly stronger than the assumptions on second order moments typically found in the stochastic gradient literature, e.g., \cite{bottou2018optimization}, as we require a uniform bound on the gradient noise. This condition is common for the algorithms using Markovian samples \citep{sun2018markov, srikant2019finite, karimi2019non}, which requires that the oscillation of stochastic gradient is controlled. For strategic classicaition problems, it is satisfied for the finite dataset setting in \Cref{ex:br}. 
Moreover, similar to \citep{doan2020finite}, this bound is adapted to the growth of $\| \theta - \theta_{PS} \|$ which is compatible with the strong convexity of the loss function $\ell(\theta; z)$. 


Our next set of assumptions pertain to the Markov kernels $\MK_{\theta} $ that generate $\{ z_k \}_{k \geq 1}$:
\begin{assumption} \label{ass:poisson}
	There exists a solution $\widehat{\grd \ell}: \RR^d \times {\sf Z} \rightarrow \RR^d$ to the Poisson equation:
	\beq \label{eq:poisson}
	\grd \ell( \theta'; z) - \grd f( \theta'; \theta ) = \widehat{\grd \ell} (\theta'; z) - \MK_{\theta} \widehat{\grd \ell} ( \theta'; z), \quad \forall~ \theta, \theta' \in \RR^d, z \in {\sf Z}.
	\eeq
\end{assumption}
\begin{assumption}\label{ass:poisson_bound}
	Consider the Poisson equation's solution $\widehat{\grd \ell} (\cdot; \cdot )$ defined in \Cref{ass:poisson}. There exists $\Lph \geq 0$ such that
	\beq \label{eq:smooth_poisson}
	\textstyle \sup_{ z \in {\sf Z} } \| \MK_{\theta} \widehat{\grd \ell}( \theta; z) - \MK_{\theta'} \widehat{\grd \ell} (\theta'; z) \| \leq \Lph \| \theta - \theta' \|, \quad \forall~\theta, \theta' \in \RR^d.
	\eeq
\end{assumption}
For Assumption~\ref{ass:poisson}, the existence of $\widehat{\grd \ell}$ in \eqref{eq:poisson} holds under mild assumptions on the Markov chains (MCs). For instance, it holds when the MC is irreducible and aperiodic, and satisfying a Lyapunov drift condition, or in the simpler case, when the MC is uniform geometrically ergodic, see \citep[Ch.~21.2]{douc2018markov}. 
In the case for performative prediction with stateful agents, the above condition holds when repeated applications of the iterative map \eqref{eq:ibr} where agents adopt their data to the current learner's model $\theta$ converges linearly to the best response. 
Assumption~\ref{ass:poisson_bound} is a smoothness condition on the kernel $\MK_\theta$ with respect to $\theta$. The condition can be satisfied when the Markov kernel is only slightly modified when $\theta$ is perturbed. 

\Cref{ass:poisson_bound} is also linked to our next assumption which is central to the study of performative prediction. Particularly, we require the distribution map $\cal D(\theta)$ to be $\epsilon$-sensitive w.r.t.~$\theta$:
\begin{assumption} \label{ass:sensitive} There exists $\epsilon \geq 0$ such that
	\beq
	W_1( {\cal D}(\theta), {\cal D}(\theta') ) = 
	\inf_{ J \in {\cal J}( {\cal D}(\theta), {\cal D}(\theta') )  }
	\EE_{ (z,z') \sim J } [ \| z - z' \|_1 ]
	\leq \epsilon \| \theta - \theta' \|,
	\eeq
	for any $\theta, \theta' \in \RR^d$. Notice that $W_1(\cdot)$ denotes the Wasserstein-1 distance and ${\cal J}( {\cal D}(\theta), {\cal D}(\theta') )$ is the set of all joint distributions on ${\sf Z} \times {\sf Z}$ with ${\cal D}(\theta), {\cal D}(\theta')$ as its marginal distribution. 
\end{assumption}
The above is a common condition for performative prediction, e.g., \citep{perdomo2020performative}. Intuitively, it allows for performative prediction algorithm to behave stably as the perturbation to ${\cal D}(\theta)$ is under control. In the subsequent analysis, we demonstrate that carefully controlling the step size in relation to $\mu, \epsilon, L$ is crucial to the convergence of Algorithm~1.



Before presenting our main result, we notice that Assumptions~\ref{ass:gradient}, \ref{ass:bounded}, \ref{ass:poisson_bound} imply that there exists constants $\overline{L}$, $\widehat{L} > 0$ such that for any $z \in {\sf Z}$, $\theta \in \RR^d$,
\beq \label{eq:corr}
 \| {\grd \ell}( \theta; z) \| \leq \overline{L} (1 + \| \theta - \theta_{PS} \| ), ~ \max\left\{ \| \widehat{\grd \ell}( \theta; z) \| , \big\|\MK_{\theta} \widehat{\grd\ell}( \theta; z ) \big\| \right\} \leq \LZ (1 + \| \theta - \theta_{PS} \| ),
\eeq
In other words, $\grd \ell(\theta; z)$, $\widehat{\grd \ell}(\theta; z)$, $\MK_\theta \widehat{\grd\ell}( \theta; z )$ are all locally bounded functions. 
Notice that $\overline{L}$ is proportional to $\sigma$ in Assumption~\ref{ass:bounded}, while $\widehat{L}$ is proportional to the maximum mixing time of the Markov chain induced by the kernel $\MK_\theta$ over all $\theta \in \RR^d$.

Our main result for the state-dependent SA algorithm is summarized as follows:
\begin{center}
\fbox{\begin{minipage}{.985\linewidth}\begin{Theorem} \label{th:main}
Under Assumptions~\ref{ass:strong}--\ref{ass:sensitive}. Suppose that the problem parameters satisfy $\epsilon < \frac{\mu}{L}$, the step sizes  $\{\gamma_{k}\}_{k\geq 1}$ are non-increasing and satisfy for any $k \geq 1$,
\beq \label{eq:stepsize_cond}
\frac{\gamma_{k-1}}{\gamma_{k}}\leq 1+ \frac{ \gamma_{k} ( {\mu} - L \epsilon )}{ 4 }, ~~ \gamma_k \leq \min \Big\{ \frac{\mu - L \epsilon}{2L^2}, \frac{ \mu - L \epsilon }{2 \Ctwo}, \frac{\min\{ (\mu - L \epsilon)/3, 3 \LZ\} }{ \Cthree + 3 \LZ ( \mu - L \epsilon) }, \frac{1}{6 \LZ} \Big\}.
\eeq 
Then for any $k \geq 1$, the expected distance between $\theta_{k}$ and the performative stable solution $\theta_{PS}$ satisfies
\beq \label{eq:main}
\EE[ \| \theta_k - \theta_{PS} \|^2 ] \leq \prod_{i=1}^k \Big( 1 - \gamma_i \frac{\mu - L \epsilon}{2} \Big) \| \theta_0 - \theta_{PS} \|^2 + \mathds{C} \!~ \gamma_k ,
\eeq
where $\EE[\cdot]$ is the expectation taken over all the randomness in \eqref{eq:sa}, \eqref{eq:mkv_setting}, and we have defined:
\beq
\mathds{C} := 
3 \LZ \overline{\Delta}  + \frac{4 \varsigma}{ \mu - L \epsilon }  \Big( 2 (2 \sigma^2 + \Cone) + (\mu - L \epsilon) \LZ + \Big( \Cthree + 3 (\mu - L \epsilon) \LZ \Big) \overline{\Delta} \Big) ,
\eeq
with $\varsigma := 1 + \gamma_1 ( \mu - L \epsilon ) / 4$, and $\Cone, \Ctwo, \Cthree$, $\overline{\Delta}$ are constants defined in \eqref{eq:cone_ctwo_main}, \eqref{eq:olDelta}, respectively.
\end{Theorem}\end{minipage}}
\end{center}
The above establishes the finite-time convergence of the studied state-dependent SA algorithm \eqref{eq:sa}, \eqref{eq:mkv_setting}. To understand this result, we observe that the step size conditions in \eqref{eq:stepsize_cond} can be satisfied by a variety of step size schedules. For instance, it can be satisfied by the constant step size $\gamma_k \equiv \gamma$; and the diminishing step size $\gamma_{k}=a_{0}/(a_{1}+k) = {\cal O}(1/k)$ with appropriate $a_0, a_1 > 0$. 
Moreover, we require the SA algorithm to work in the regime when $\epsilon < \mu / L$. This is similar to \cite{perdomo2020performative} which ensures that the solution $\theta_{PS}$ is a stable fixed point to \eqref{eq:ps}.

The main result is stated on the expected squared distance of $\| \theta_k - \theta_{PS} \|^2$ in \eqref{eq:main}. Here, the bound consists of a \emph{transient} term and a \emph{fluctuation} term. The \emph{transient} term decays sub-exponentially as ${\cal O}( \exp( - \frac{\mu-L\epsilon}{2} \sum_{i=1}^k \gamma_i ) )$ and is scaled by the initial gap $\| \theta_0 - \theta_{PS} \|^2$. The \emph{fluctuation} term is in the order of ${\cal O}(\gamma_k)$ and is scaled by $\mathds{C}$ which depends on the oscillation of stochastic gradient (via $\sigma$) and the mixing time of the controlled Markov chain (via $\LZ$). 
With a diminishing step size schedule such as $\gamma_{k}= c_{0}/( c_{1}+k)$, \Cref{th:main} shows that the state dependent SA algorithm finds the performative stable solution $\theta_{PS}$ at the rate of ${\cal O}(1/k)$ in expectation.  

\paragraph{On Non-strongly-convex Loss Function} An obvious shortcoming with \Cref{th:main} is the requirement of strongly convex loss functions [cf.~\Cref{ass:strong}]. Below, we comment on the convergence of the state-dependent SA algorithm \eqref{eq:sa}, \eqref{eq:mkv_setting} when the loss function is not strongly convex (possibly non-convex). Notice that in the absence of \Cref{ass:strong}, the performative stable solution $\theta_{PS}$ may not be well defined. We resort to finding a stationary point to the performative prediction problem \eqref{eq:performative}.

Our idea is to view \eqref{eq:sa}, \eqref{eq:mkv_setting} as a biased SA algorithm with mean field $h(\theta) = \EE_{z \sim {\cal D}(\theta)} [ \grd \ell( \theta; z )] = \grd f(\theta; \theta)$. This mean field is correlated with the gradient for the performative loss in \eqref{eq:performative}. Under additional assumptions on $\ell(\theta;z)$, ${\cal D}(\theta)$, in Appendix~\ref{app:noncvx} we show
\beq \label{eq:biased}
\pscal{ h(\theta) }{ \grd \EE_{ z \sim {\cal D}(\theta) } [ \ell( \theta; z ) ] } \geq \| h(\theta) \|^2 / 2 - {\rm c}_0 ,~~\forall~\theta \in \RR^d,
\eeq
where ${\rm c}_0 > 0$ is a bias dependent on the sensitivity of distribution shift [cf.~\Cref{ass:sensitive}].
We define the constant ${\rm c}_0$ in \eqref{app:c_0} of Appendix~\ref{app:noncvx}, which is shown to be dependent on $\ell(\theta; z)$,  $\grd_{\theta} \log(p_{\mathcal{D}(\theta)}(z))$. 

It should be pointed that the algorithm \eqref{eq:sa} may not provide a `good' solution to the non-convex performative learning problem \eqref{eq:performative}. However, if the  state-dependent distribution is not too sensitive to the change of state, the following corollary shows that \eqref{eq:sa} would still converge to a ${\rm c}_{0}$-neighborhood of a stationary solution. Before we discuss the main statement, we need two additional assumptions: 
\begin{manualtheorem}{2'} \label{ass:twoprime}
    The function $V(\theta) = \EE_{ z \sim {\cal D}(\theta) } [ \ell( \theta; z ) ]$ is continuously differentiable and there exists $L_V \geq 0$ such that $\| \grd V(\theta) - \grd V(\theta') \| \leq L_V \| \theta - \theta' \|$ for any $\theta, \theta' \in \RR^d$. 
\end{manualtheorem}
\begin{manualtheorem}{3'} \label{ass:threeprime}
    There exists $\sigma \geq 0$ such that $\| \grd \ell(\theta; z) - \grd f(\theta; \theta) \| \leq \sigma$ for any $\theta \in \RR^d, z \in {\sf Z}$.
\end{manualtheorem}
The above are stronger conditions than Assumptions~\ref{ass:gradient}, \ref{ass:bounded}, yet are reasonable settings for certain non-convex loss functions. 
For instance, as we will show in Appendix \ref{app:noncvx}, Assumption \ref{ass:twoprime} holds if $\ell(\theta ; z), \nabla_{\theta} \log \left(p_{D(\theta)}(z)\right)$ are bounded, and $\nabla_{\theta} \log \left(p_{D(\theta)}(z)\right)$ to be Lipschitz w.r.t. $\theta$, e.g., when $\mathcal{D}(\theta)$ is `smooth', e.g., it is Gaussian or softmax distribution. 
Assumption \ref{ass:threeprime} holds under similar condition as Assumption \ref{ass:bounded}, e.g., ${\sf Z}$ is compact. 
We obtain the following as a corollary of \citep[Theorem 2]{karimi2019non}:
\begin{center}
\fbox{\begin{minipage}{.985\linewidth}
\begin{Corollary} \label{cor:ncvx}
Under Assumptions~\ref{ass:twoprime}, \ref{ass:threeprime}, \ref{ass:poisson}, \ref{ass:poisson_bound}, and let  \eqref{eq:biased} holds. With a step size sequence that decays in the order of $\gamma_k = {\cal O}(1/\sqrt{k})$, it holds for any $K \geq 1$ that
	\beq \textstyle 
	\EE[ \| \grd V( \theta_{\sf K} ) \|^2 ] = {\cal O} ( \log K / \sqrt{K} + {\rm c}_0 ),~~\PP ( {\sf K} = k ) = \gamma_k / \sum_{j=1}^K \gamma_j,~\forall~k \in \{1,...,K\}.
	\eeq
Note that ${\sf K} \in \{1,...,K\}$ is a discrete r.v.~independent of the randomness in the SA algorithm and $\EE[\cdot]$ denotes the total expectation. 
\end{Corollary}\end{minipage}}\end{center}
See the details in Appendix~\ref{app:noncvx}. 
The above corollary shows that even without the strong convexity assumption, the state dependent SA algorithm finds an ${\cal O}(\log K / \sqrt{K} + {\rm c}_0)$-stationary solution to the performative prediction problem \eqref{eq:performative} in at most $K$ iterations.

\section{Proof Outline of \Cref{th:main}} \label{sec:pf}
We outline the main steps in proving \Cref{th:main}. Our proof strategy consists in tracking the progress of the mean squared error $\Delta_k \eqdef \EE[ \| \theta_k - \theta_{PS} \|^2 ]$. 
To simplify notations, we define $\tilde{\mu} \eqdef \mu - L \epsilon$, the scalar product $G_{m:n} = \prod_{i=m}^n (1 - \gamma_i \tilde{\mu})$, for $n > m \geq 1$, and $G_{m:n} = 1$ if $n \leq m$. 

The following lemma describes the one-step progress of the SA algorithm. 
\begin{Lemma} \label{lem:square}
    Under Assumptions~\ref{ass:strong}, \ref{ass:gradient}, \ref{ass:bounded}, \ref{ass:sensitive}. For any $k \geq 0$, it holds
    \beq \label{eq:square_1st}
    \begin{split}
    \| \theta_{k+1} - \theta_{PS} \|^2 & \leq \big( 1 - 2 \gamma_{k+1} \tilde{\mu} + 2L^2 \gamma_{k+1}^2 \big) \| \theta_{k} - \theta_{PS} \|^2 \\
    & \quad + 2 \sigma^2 \gamma_{k+1}^2 - 2 \gamma_{k+1} \pscal{ \theta_k - \theta_{PS} }{ \grd \ell( \theta_k; z_{k+1} ) - \grd f( \theta_k; \theta_k ) } .
    \end{split}
    \eeq
\end{Lemma}
The proof can be found in Appendix~\ref{app:square}, which involves a simple expansion of the squared error.
The above lemma suggests that the sensitivity parameter shall satisfy $\epsilon < {\mu} / {L}$ to ensure $\tilde{\mu} > 0$. Furthermore, the step size condition 
$\sup_{k \geq 1} \gamma_{k} \leq {\tilde{\mu}} / (2 L^2)$ in \eqref{eq:stepsize_cond} leads to $1 - 2 \gamma_{k+1} \tilde{\mu} + 2L^2 \gamma_{k+1}^2 \leq 1 - \gamma_{k+1} \tilde{\mu}$ such that the first term in the r.h.s.~of \eqref{eq:square_1st} is a contraction. 
Under the above premises and suppose $z_{k+1} \overset{\text{i.i.d.}}{\sim} {\cal D}( \theta_k )$ as in the case of \citep{mendler2020stochastic, drusvyatskiy2020stochastic}, the stochastic gradient in \eqref{eq:square_1st} is conditionally \emph{unbiased}. As such, \Cref{lem:square} leads to the recursion $\Delta_{k+1} \leq ( 1 - \gamma_{k+1} \tilde{\mu} ) \Delta_k + 2 \sigma^2 \gamma_{k+1}^2$, implying $\Delta_k = {\cal O}(\gamma_k)$.

However, for the state-dependent SA algorithm  \eqref{eq:sa}, \eqref{eq:mkv_setting},  the stochastic gradient $\grd \ell( \theta_k; z_{k+1} )$ is conditionally \emph{biased} and is driven by a controlled MC. Under the stepsize condition $\sup_{k \geq 1} \gamma_{k} \leq {\tilde{\mu}} / (2 L^2)$, taking the total expectation and solving the recursion in \eqref{eq:square_1st} yield
\beq \label{eq:inter}
\begin{split}
\Delta_{k} & \textstyle \leq G_{1:k} \Delta_0 + 2 \sigma^2 \sum_{s=0}^{k-1} G_{s+2:k} \gamma_{s+1}^2 \\
& \textstyle \quad + 2 \sum_{s=0}^{k-1} G_{s+2:k} \gamma_{s+1} \EE\big[ \pscal{ \theta_{PS} - \theta_s }{ \grd \ell( \theta_s; z_{s+1}) - \grd f(\theta_s; \theta_s)  } \big] .
\end{split}
\eeq
It can be shown that the first two terms are bounded by ${\cal O}(\gamma_k)$. We are interested in analyzing the last term when the samples $\{ z_{k} \}_{k \geq 1}$ are drawn according to \eqref{eq:mkv_setting}. Observe that:
\begin{Lemma} \label{lem:poisson}
    Under Assumptions~\ref{ass:gradient}--\ref{ass:sensitive} and the stepsize conditions in \eqref{eq:stepsize_cond}. For any $k \geq 1$, it holds
	\beq \notag
	\begin{split}
	&  2 \sum_{s=0}^{k-1} G_{s+2:k} \gamma_{s+1} \EE\big[ \pscal{ \theta_{PS} - \theta_s }{ \grd \ell( \theta_s; z_{s+1}) - \grd f(\theta_s; \theta_s)  } \big] \leq \\
	& \sum_{s=2}^k \gamma_s^2 G_{s+1:k} \big( \Cone + \Ctwo \Delta_{s-1} + \Cthree \Delta_{s-2} \big) + \gamma_1 G_{2:k} \big\{ \LZ (1 + 3 \Delta_0 ) + \gamma_1 \Cone \big\} + \gamma_k \LZ \big\{ 1 + 3 \Delta_{k-1} \big\},
	\end{split}
	\eeq
	where we have defined the constants:
	\beq \label{eq:cone_ctwo_main}
	\Cone := \varsigma \Lph \overline{L} + 4 \varsigma \overline{L} \LZ + ( 1 + \tilde{\mu} ) \varsigma \LZ, ~~\Ctwo := 2 \varsigma \Lph \overline{L},~~
	\Cthree := \Cone + 2 ( 1 + \tilde{\mu} )\varsigma \LZ.
	\eeq
\end{Lemma}
The analysis is inspired by \citep{benveniste2012adaptive} and has been adopted in recent works such as \citep{atchade2017perturbed,karimi2019non}; see the details in Appendix~\ref{app:poisson}. To handle the controlled MC, our technique involves applying the Poisson equation in \Cref{ass:poisson} and decomposing the gradient error $\grd \ell( \theta_s; z_{s+1}) - \grd f(\theta_s; \theta_s)$ into Martingale and finite difference terms. 

A key difference between \Cref{lem:poisson} and analysis in the previous works is that the latter assumed that the stochastic gradients $\grd \ell(\theta_k; z_{k+1} )$ are bounded which greatly simplifies the proof. Our assumptions are significantly weaker as the latter actually grows as ${\cal O}( 1 + \| \theta_k - \theta_{PS} \|)$ [cf.~\Cref{ass:bounded}]. 
From the analysis standpoint, this demands a new proof technique as we present next. 

To proceed, observe that substituting \Cref{lem:poisson} into \eqref{eq:inter} yields:\vspace{-.1cm}
\beq
\begin{split}
	\Delta_{k} & \textstyle \leq G_{1:k} \Delta_0 + \sum_{s=1}^{k-1} \gamma_{s+1}^2 G_{s+2:k} \big( 2 \sigma^2 + \Cone + \Ctwo \Delta_{s} + \Cthree \Delta_{s-1} \big) \\
	& \quad + \gamma_1 G_{2:k} \big\{ \LZ (1 + 3 \Delta_0 ) + \gamma_1 ( 2 \sigma^2 + \Cone ) \big\} + \gamma_k \LZ \big\{ 1 + 3 \Delta_{k-1} \big\}.
\end{split}
\eeq
With the first step size condition in \eqref{eq:stepsize_cond}, we can apply the auxiliary result in \Cref{lem:o_gamma_k} from the appendix, which simplifies the upper bound as\vspace{-.1cm}
\beq \label{eq:lem5}
\begin{split}
	\Delta_{k} & \textstyle  \leq G_{1:k} \Delta_0 + \Big( \frac{ 2}{ \tilde{\mu} } (2 \sigma^2 + \Cone) + \LZ \Big) \gamma_k + \sum_{s=1}^{k-1} \gamma_{s+1}^2 G_{s+2:k} \big( \Ctwo \Delta_{s} + \Cthree \Delta_{s-1} \big) \\
	& \quad + \gamma_1 G_{2:k} \big\{ \LZ (1 + 3 \Delta_0 ) + \gamma_1 ( 2 \sigma^2 + \Cone ) \big\} + 3 \gamma_k \LZ \Delta_{k-1} .
\end{split}
\eeq
Observe that the first row in \eqref{eq:lem5} is already in a similar form to the bound presented in the theorem. The key issue lies with the last term $3 \gamma_k \LZ \Delta_{k-1}$ which may be unbounded. We show that our choice of step sizes in \eqref{eq:stepsize_cond} ensures the convergence of $\Delta_k$ to ${\cal O}(\gamma_k)$:
\begin{Lemma} \label{lem:bdd}
	Suppose that $\{ \Delta_k \}_{k \geq 0}$ satisfy \eqref{eq:lem5} and the step sizes $\{ \gamma_k \}_{k \geq 1}$ satisfy \eqref{eq:stepsize_cond}. It holds {\sf (i)}\vspace{-.1cm}
	\beq \label{eq:olDelta}
	\sup_{k \geq 0} \Delta_k \leq \overline{\Delta} := 3 \Delta_0 + \frac{\varsigma}{ 9 \LZ^2 } \Big( 2 (2 \sigma^2 + \Cone) + (\mu - L \epsilon) \LZ \Big), \vspace{-.2cm}
	\eeq
    and {\sf (ii)} the following inequality holds for any $k \geq 1$:
	\beq \label{eq:bdDelta}
	\Delta_k \leq \prod_{i=1}^k ( 1 - \gamma_i \frac{\tilde{\mu}}{2}) \Delta_0 + \Big\{ 3 \LZ \overline{\Delta}  + \frac{4 \varsigma}{ \tilde{\mu}}  \Big( 2 (2 \sigma^2 + \Cone) + \tilde{\mu} \LZ + \Big( \Cthree + 3 \LZ \tilde{\mu} \Big) \overline{\Delta} \Big) \Big\} \gamma_k. 
	\eeq
	\vspace{-.6cm}
\end{Lemma}
Proving the above lemma requires one to establish the stability of the system \eqref{eq:lem5}, which demands a sufficiently small $\gamma_k$ to control the remainder term $3 \LZ \gamma_k \Delta_{k-1}$. Our analysis relies on the special structure of this inequality system; see the proof details in Appendix~\ref{app:bdd}. 
The convergence bound \eqref{eq:bdDelta} follows from the boundedness of $\Delta_k$. 
Finally, we obtain \Cref{th:main} by applying \Cref{lem:bdd}.

\section{Numerical Experiments} \label{sec:num}
This section considers two performative prediction problems to corroborate with our theories. All the experiments are performed with Python on a server using a single thread of an Intel Xeon 6138 CPU. Further details about the experiments below can be found in Appendix~\ref{app:details}.

\textbf{Gaussian Mean Estimation}~~
The \emph{first problem} is concerned with Gaussian mean estimation using synthetic data. Our aim is to validate Theorem~\ref{th:main} using a simple experiment. Here, \eqref{eq:performative} is specified as $\min_{ \theta \in \RR } \EE_{ z \sim {\cal D}(\theta)} [ (z-\theta)^2/2 ]$ with ${\cal D}(\theta) \equiv {\cal N}( \bar{z} + \epsilon \theta ; \sigma^2 )$. For $0<\epsilon<1$, the performative stable solution has a closed form $\theta_{PS} = \frac{ \bar{z} }{ 1 - \epsilon }$. For the state-dependent SA, the \emph{agent} follows an autoregressive (AR) model $z_{k+1} = (1 - \rho) z_k + \rho \tilde{z}_{k+1}$ with independent $\tilde{z}_{k+1} \sim {\cal N}( \bar{z} + \epsilon \theta_k; \sigma^2 )$ and regression parameter $\rho \in (0,1)$. This AR recursion is a controlled MC with a stationary distribution that yields the unbiased gradient of \eqref{eq:performative}, details about the MC are in Appendix~\ref{app:gau}.

We consider a large variance setting with $\bar{z} = 10$, $\sigma = 50$, $\epsilon = 0.1$. The step size is $\gamma_k = \frac{c_0}{c_1+k}$, $c_0 = \frac{500}{\tilde{\mu}}, c_1 = \frac{800}{\tilde{\mu}^2}$. In Fig.~\ref{fig:main_sim} (left), we compare $|\theta_k - \theta_{PS}|^2$ against the iteration number $k$ for the Gaussian estimation problem using our state-dependent SA and greedy deploy \citep{mendler2020stochastic} algorithms. As observed, both algorithms have an asymptotic convergence rate of ${\cal O}(1/k)$ towards $\theta_{PS}$ which is predicted by \Cref{th:main}. As $\rho \downarrow 0$, the state-dependent SA algorithm delivers a smaller error as the AR model has a stationary distribution with lower variance. 

\begin{figure}
    \centering
    \includegraphics[width=.33\linewidth]{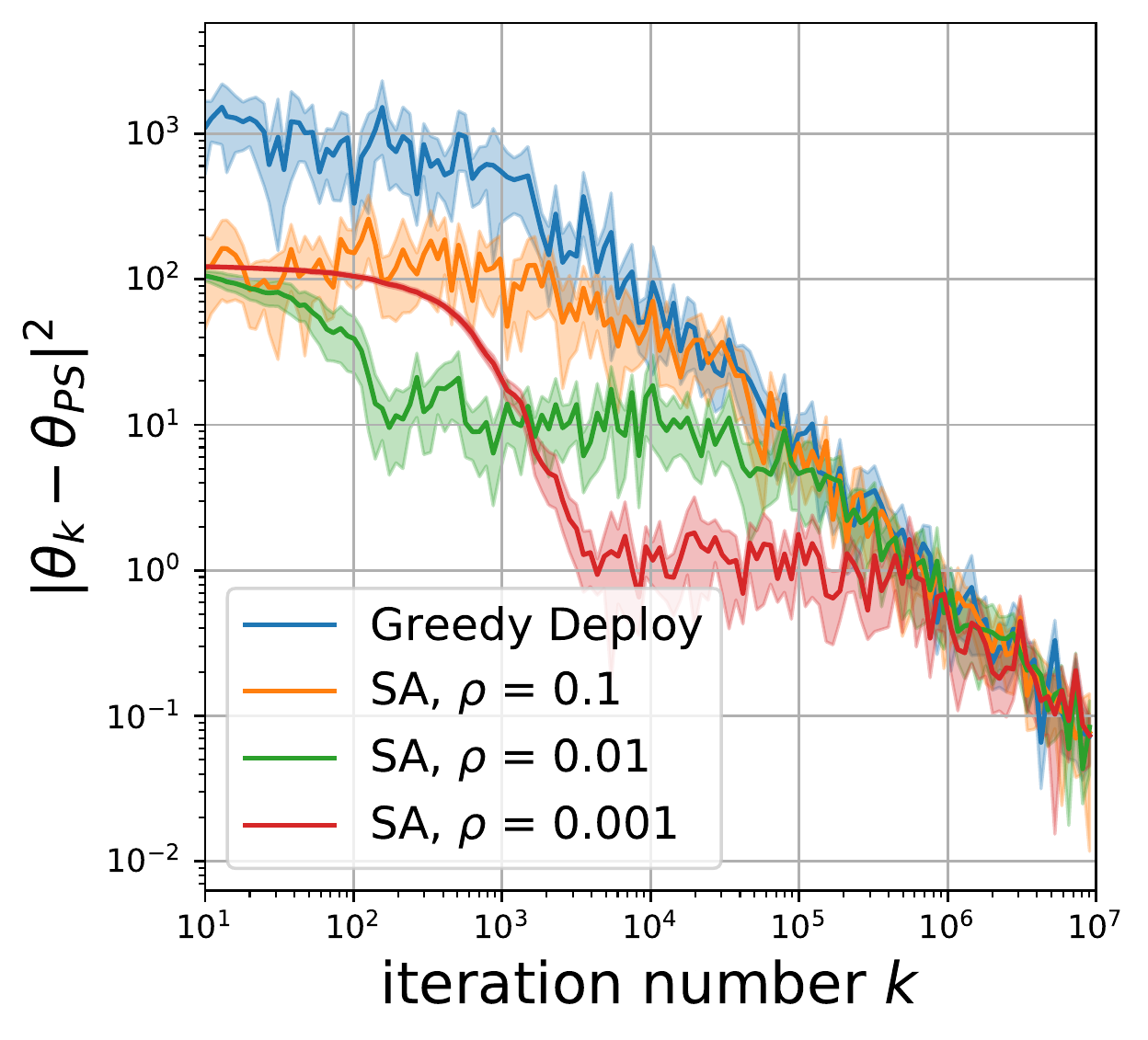}~\includegraphics[width=.33\linewidth]{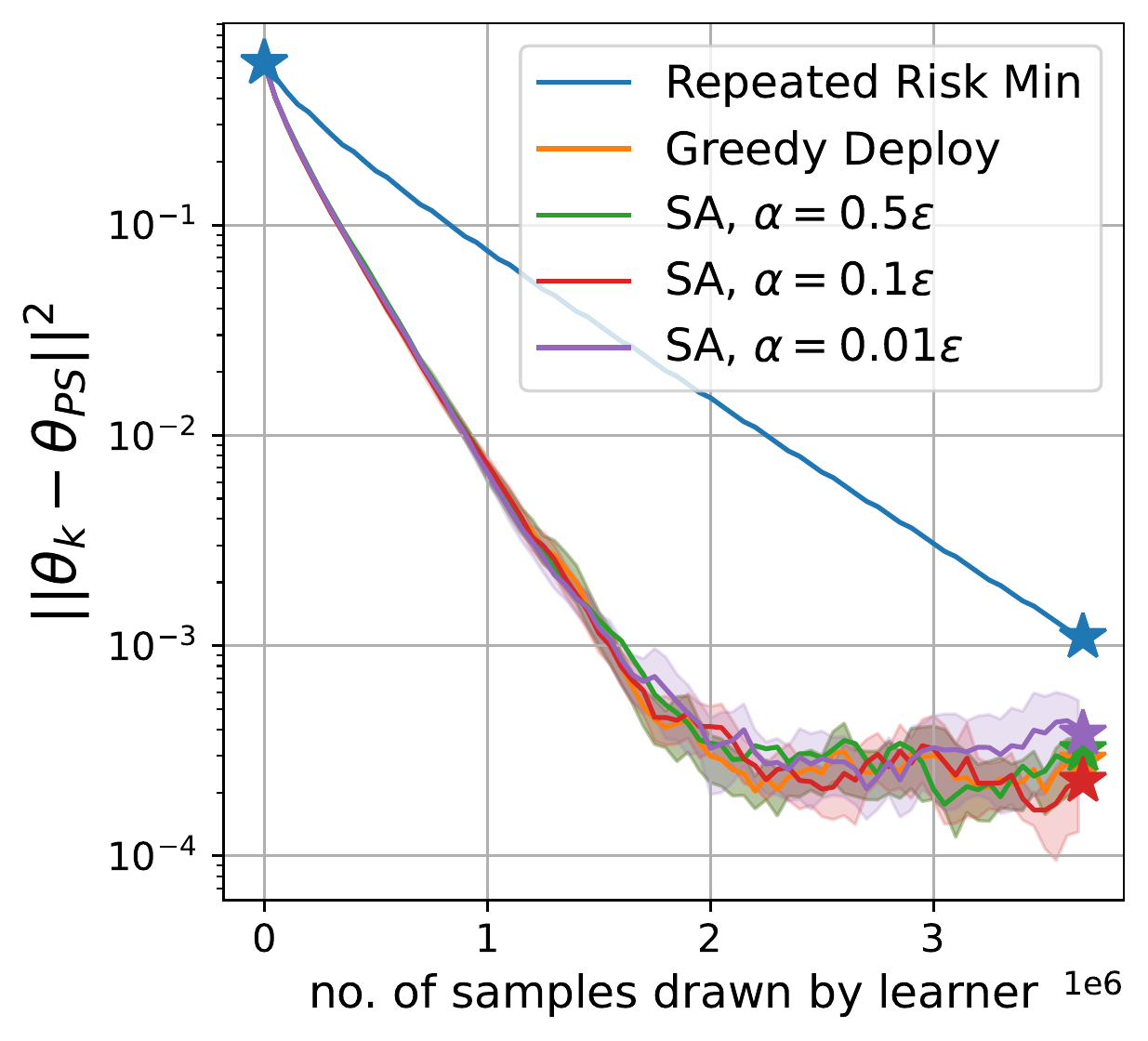}~\includegraphics[width=.33\linewidth]{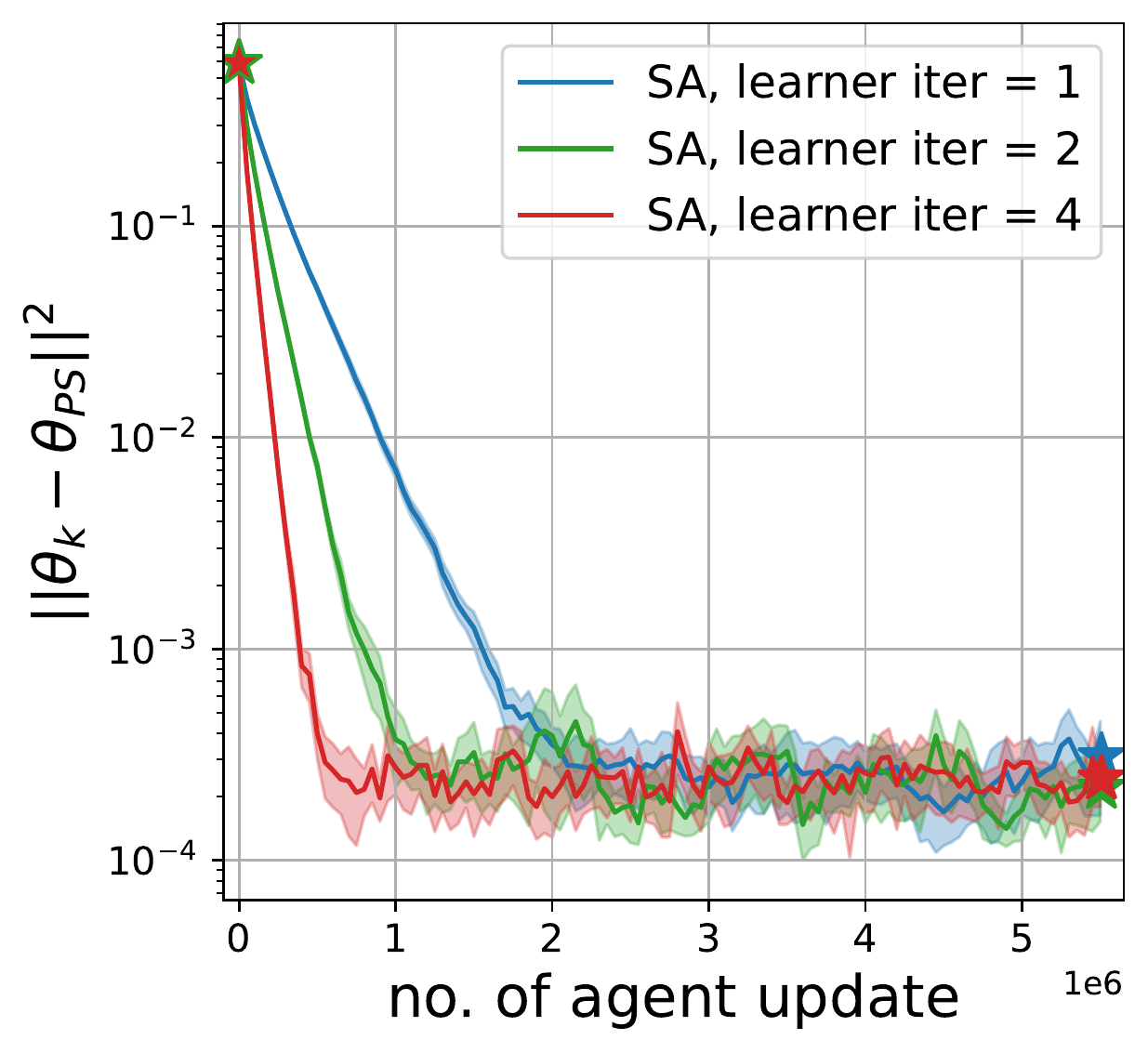}\vspace{-.6cm}
    \caption{\textbf{Gaussian mean estimation} -- \emph{(Left)} Under different regression parameter $\rho$; \textbf{Strategic Classification} -- \emph{(Middle)} Under Linear BR $U_{\sf q}(\cdot)$ and different agent response rate $\alpha$ [cf.~\eqref{eq:ibr}]; \emph{(Right)} Under Logistics BR $U_{\sf lg}(\cdot)$. The shaded region shows the 90\% confidence interval over 20 trials.}\vspace{-.2cm}
    \label{fig:main_sim}
\end{figure}

\textbf{Strategic Classification}~~
The \emph{second problem} is a strategic classification (SC) problem similar to  \cite{perdomo2020performative} 
for a credit scoring classifier with \texttt{GiveMeSomeCredit} dataset\footnote{Available at \url{https://www.kaggle.com/c/GiveMeSomeCredit/data}.}.
Our aim is to showcase the effects when the agent adapts slowly to the shifted distribution. In particular, our theory implies that while the algorithm will still converge to $\theta_{PS}$, a slower convergence rate will be observed. To specify \eqref{eq:performative}, let $z \equiv (x,y)$ where $x \in \RR^d$ is feature vector, $y \in \{0,1\}$ is label. The \emph{learner} finds $\theta \in \RR^d$ that minimizes:
\beq \textstyle \label{eq:obj_cred}
\EE_{z \sim {\cal D}(\theta)}[ \ell( \theta; z) ], ~~\text{where}~~ \ell( \theta; z ) = \frac{\beta}{2} \| \theta \|^2 + \log( 1 + \exp( \pscal{ \theta }{x} ) ) - y \pscal{\theta}{x}.
\eeq
Observe that when $\beta > 0$, $\ell(\theta;z)$ is a $\beta$-strongly convex function w.r.t.~$\theta$ satisfying \Cref{ass:strong}. 
For any $\theta \in \RR^d$, the shifted data distribution ${\cal D}(\theta)$ is obtained through evaluating the best response (BR) in \eqref{eq:br} of \textbf{Example~\ref{ex:br}}. We consider two types of utility functions adopted by the \emph{agents}:
\beq \notag \textstyle
    U_{\sf q}(z^{\prime};{z}, \theta) = \pscal{\theta}{x'}- \frac{\|x'-x\|^{2}}{2\epsilon},~~ U_{\sf lg}(z^{\prime};{z}, \theta) = y \pscal{\theta}{x'} - \log\left(1+\exp(\pscal{\theta}{x'})\right)- \frac{\|x'-x\|^{2}}{2\epsilon},
\eeq
where $z \equiv (x,y)$ is the original unshifted data. The label $y \in \{0,1\}$ is unchanged in the BR. Notice that $U_{\sf q}(\cdot)$, $U_{\sf lg}(\cdot)$ have respectively linear and logistics costs. Both utility functions include a quadratic regularizer where $\epsilon$ controls the sensitivity of the distribution shift [cf.~\Cref{ass:sensitive}].

With a published $\theta_k$, the agent(s) maximize the utility function prior to giving data to the learner for the next round. 
For both $U_{\sf q}(\cdot)$ and $U_{\sf lg}(\cdot)$, the BR obtained steers the classifier in favor of the agent(s).
Furthermore,  $U_{\sf lg}(\cdot)$ is motivated by logistic regression which favors towards samples with label `1'. 

For details of the numerical experiments, we set $\beta = 1000/m$ in \eqref{eq:obj_cred}, $\epsilon = 0.01$ in the utility functions, and  in \eqref{eq:ibr}, we set number of selected agents as $|{\cal I}_k| = 5$, agents' response rate (stepsize) as $\alpha = 0.5 \epsilon$ unless otherwise specified. 
The step size for \eqref{eq:sa} is $\gamma_k = {c_0} / {(c_1+k)}$, $c_0 = {100} / {\tilde{\mu}}, c_1 = { 8 L^2 } / {\tilde{\mu}^2}$. 

We first consider when $\theta_{PS}$ is computed with ${\cal D}(\theta)$ defined by the linear BR function $U_{\sf q}(\cdot)$ and compare our state-dependent SA with greedy deploy scheme \citep{mendler2020stochastic} and repeated risk minimization \citep{perdomo2020performative}.  As shown in Fig.~\ref{fig:main_sim} (middle), all algorithms converge to $\theta_{PS}$. As $\alpha \downarrow 0$, the state-dependent SA converges at slightly slower rates as the agents adapt to the distribution shift with increased mixing time of the MC.
The result corroborates with \Cref{th:main} which established the ${\cal O}(1/k)$ convergence rate with state-dependent SA.

Our last experiment pertains to the same SC problem as before, but we consider a different setting where $\theta_{PS}$ is computed  with ${\cal D}(\theta)$ defined by the logistics BR function $U_{\sf lg}(\cdot)$.
The agents follow a more complicated dynamics since the BR does not admit a closed form solution. Again, we aim to confirm Theorem~\ref{th:main} that Algorithm~1 converges under the stateful agent setting. We compare the distance $\| \theta_k - \theta_{PS}\|^2$ versus iteration $k$. 
In addition to showing the convergence result as predicted in Theorem~\ref{th:main}, we aim to observe the effects when the learner adopts a lazy deployment scheme [cf.~\cite{mendler2020stochastic}], by varying the number of $\theta$-update by the \emph{learner} per the agents' update. 
Fig.~\ref{fig:main_sim} (right) shows the error $\| \theta_k - \theta_{PS} \|^2$ against the number of adaptation steps performed at the agents via \eqref{eq:ibr} as we illustrate the convergence rate from the perspectives of the agents.
We observe that the error decreases at a faster rate when the number of learner's iteration increases. More details about our numerical experiments can be found in Appendix~\ref{app:details}.

\textbf{Conclusions \& Limitations}~~~ We consider a state-dependent SA algorithm for performative prediction. We showed a convergence rate of ${\cal O}(1/k)$ in mean-squared error towards the performative stable solution when the agents provide data drawn from a controlled MC. Our study paved the first step towards understanding and applying performative prediction in a dynamical setting.  

There are several limitations. First, despite being a reasonable fixed point, $\theta_{PS}$ can be different from the optimal solution to \eqref{eq:performative}. An open problem is whether the state dependent SA converges when gradient correction, e.g., \citep{izzo2021learn}, is used. Second, our convergence analysis of state-dependent SA algorithms relies on smoothness conditions on the controlled MC and distributions [cf.~\Cref{ass:poisson_bound}, \ref{ass:sensitive}]. An open problem is to verify if these are necessary for our results to hold.

\bibliographystyle{abbrvnat}
\bibliography{main}

\appendix 
\newpage
\section{Supplementary Information for Section~\ref{sec:main}}

\subsection{Example on Gaussian Estimation (A Case where $\pi_{\theta}(\cdot) \neq {\cal D}(\theta)$)}\label{app:gau}
Consider the following instance of \eqref{eq:performative} with:
\beq \label{eq:gau}
\min_{ \theta \in \RR }~ \EE_{ z \sim {\cal D}(\theta) } [ ( z - \theta)^2 / 2 ] \quad \text{where} \quad {\cal D}(\theta) \equiv {\cal N}( \bar{z} + \epsilon \theta ; \sigma^2 ).
\eeq
Following \eqref{eq:sa}, the state-dependent SA algorithm reads
\beq \label{eq:sa_gau}
\theta_{k+1} = \theta_k - \gamma_{k+1} \grd \ell(\theta_k; z_{k+1}) = \theta_k - \gamma_{k+1} ( \theta_k - z_{k+1} ).
\eeq
where the sequence $\{ z_k \}_{k \geq 1}$ is generated by an autoregressive (AR) model, with $\rho \in (0,1]$,
\beq \label{eq:ar}
z_{k+1} = (1 - \rho) z_k + \rho \tilde{z}_{k+1} \quad \text{where} \quad \tilde{z}_{k+1} \sim {\cal D}( \theta_k ) = {\cal N}( \bar{z} + \epsilon \theta_k ; \sigma^2 ),
\eeq
such that the draw of $\tilde{z}_{k+1}$ are independent.
We show that the algorithm \eqref{eq:sa_gau}, \eqref{eq:ar} can be analyzed as a state-dependent SA \eqref{eq:sa}, \eqref{eq:mkv_setting} considered in our framework. Precisely, we show that the controlled MC in \eqref{eq:ar} admits a stationary distribution $\pi_{\theta_k}(\cdot)$ such that $\EE_{z \sim \pi_{\theta_k}(\cdot)} [ \grd \ell(\theta; z) ] = \EE_{z' \sim {\cal D}(\theta_k)} [ \grd \ell( \theta ; z' )]$.

Observe that \eqref{eq:ar} defines a controlled MC with a transition kernel denoted by $\MK_{\theta_k}: \RR \times \RR \rightarrow \RR_+$ at the $k$th iteration, as in \eqref{eq:mkv_setting}. For every $\theta \in \RR$, $z \in \RR$, the kernel $\MK_{\theta}$ has a unique stationary distribution given by
\beq \label{eq:gau_stat}
\lim_{ n \rightarrow \infty } \MK_{\theta}^n (z, \cdot ) = \pi_{\theta}(\cdot) \equiv {\cal N} \big( \bar{z} + \epsilon \theta \, ; \, \frac{\rho}{2-\rho} \sigma^2 \big).
\eeq
Notice that the above is different from the distribution ${\cal D}(\theta)$ desired in \eqref{eq:gau} unless $\rho = 1$. In the latter case, the AR model \eqref{eq:ar} reduces to drawing i.i.d.~samples from ${\cal D}(\theta)$. For general $\rho < 1$, it still satisfies the asymptotically unbiasedness of the stochastic gradient estimate. In particular,
\beq
\EE_{ z \sim \pi_{\theta}(\cdot) } [ \grd \ell( \theta ; z ) ] = \EE_{ z \sim \pi_{\theta}(\cdot) } [ \theta - z ] = \EE_{ z \sim {\cal D}(\theta) } [ \grd \ell( \theta ; z ) ] .
\eeq
The key observation is that for this particular performative prediction problem \eqref{eq:gau}, the gradient of the loss function is \emph{linear} in the sample $z$. As such, with \eqref{eq:gau_stat} yielding a stationary distribution that has the same mean as ${\cal D}(\theta)$, the asymptotic unbiasedness property is unaffected.
In fact, the stationary distribution in \eqref{eq:gau_stat} has a reduced variance compared to ${\cal D}(\theta)$. Therefore, we expect the estimation error of $\theta_{PS}$ to be more stable using \eqref{eq:sa_gau}, \eqref{eq:ar} than \citep{mendler2020stochastic}.  

\subsection{Details on Example~\ref{ex:br} for Adapted Best Response}\label{app:br}
We continue the discussions in the paper with the procedure \eqref{eq:ibr}.
When $\theta \in \RR^d$ is fixed, the procedure in \eqref{eq:ibr} is modelled as a Markov Chain (MC) with unique stationary distribution that corresponds to the best response distribution ${\cal D}(\theta)$ described in \eqref{eq:br}. 

To this end, we model the state of the MC by the tuple $\hat{z} \equiv ( d_1, ..., d_m, z )$. Consider the state space given by ${\sf Z}^{m+1}$ and
denote $\MK_{\theta} : {\sf Z}^{m+1} \times {\cal Z}^{m+1} \rightarrow \RR_+$ as the Markov transition kernel.
We remark that there is a slight abuse of notation here as the stochastic gradient $\grd \ell( \theta; \hat{z} )$ used by the learner depends only on the last term, $z$, in the agents' state variable $\hat{z}$. We have decided to use the current notation in the main paper to avoid introducing complicated notation for the implementation focused readers. Nevertheless, the SC example fits our proposed model.

Turning back on the MC. Observe when the current state is $\hat{z}$, under the action of kernel $\MK_{\theta}$, by following the description in \eqref{eq:ibr}, we obtain the next state $\hat{z}' = ( d_1', ..., d_m', z' )$ as 
\beq \label{eq:mc_upd}
d_j' = d_j + \alpha \mathds{1}_{ j \in {\cal I} }\grd U( d_j; \bar{d}_j, \theta ),~j=1,...,m, \quad z' = d_i',
\eeq
with probability 
\[ 
	\frac{1}{ { m \choose pm } } \times \frac{1}{m} ,
\]
for any ${\cal I} \subseteq \{1,...,m\}$, $|{\cal I}| = pm$ and $i \in \{1,...,m \}$.

At each transition, the data points $\{ d_1,..., d_m \}$ are updated by the first equation in \eqref{eq:mc_upd}. The latter can be treated as a special case of the \emph{random block coordinate gradient descent (RBCD)} algorithm to the \emph{separable} problem:
\beq \label{eq:rbcd}
\max_{ d_i , i =1,...,m}~\sum_{i=1}^m U( d_i ; \bar{d}_i, \theta ).
\eeq
Note that the optimal solution to the above, $\{ d_1^\star, ..., d_m^\star \}$, is a set of data points that forms the empirical distribution ${\cal D}(\theta)$.
Furthermore, it is known that the RBCD algorithm converges linearly with high probability and almost surely to the optimal solution for strongly concave maximization; see \cite{richtarik2014iteration, patrascu2015efficient}. 

With the above observations, the MC induced by $\MK_\theta$ has a stationary distribution $\pi_\theta (\cdot)$ where for any measurable function $f : {\sf Z}^{m+1} \rightarrow \RR^n$, it holds
\beq
\lim_{k \rightarrow \infty} \MK_\theta^k f( d_1,..., d_m, z ) = \frac{1}{m} \sum_{i=1}^m f( d_1^\star, ..., d_m^\star, d_i^\star ) = \EE_{z' \sim {\cal D}(\theta)} [ f( d_1^\star, ..., d_m^\star , z') ],
\eeq
for any initial state $\hat{z} = ( d_1, ..., d_m, z )$. The above identity can be derived from the fact that the RBCD algorithm converges almost surely to the optimal solution to \eqref{eq:rbcd} and the random variable $z^k$ is uniformly drawn from $\{ d_1^k, ..., d_m^k \}$. 
Furthermore, for any $L$-Lipschitz continuous $f$, it holds
\beq
\begin{split}
& \Big\| \MK_\theta^k f( d_1,..., d_m, z ) - \EE_{z' \sim {\cal D}(\theta)} [ f( d_1^\star, ..., d_m^\star , z') ] \Big\| \\
& \overset{(a)}{\leq} \frac{1}{m} \sum_{i=1}^m \| \EE[ f( d_1^k, ..., d_m^k, d_i^k )] - f( d_1^\star, ..., d_m^\star, d_i^\star ) \| \\
& \overset{(b)}{\leq} L \big(1 + \frac{1}{m} \big) \sum_{i=1}^m \EE[ \| d_i^k - d_i^\star \| ] \leq \overline{\rm C} \rho^k,
\end{split}
\eeq
where (a) uses $\MK_\theta^k f( d_1,..., d_m, z ) = \sum_{i=1}^m \EE[ f( d_1^k , ..., d_m^k , d_i^k )] / m$ and the expectation is taken with respect to the random subset selection of ${\cal I}_k$ in \eqref{eq:ibr}. In the expression that follows (b), the constants $\overline{\rm C}$, $\rho \in [0,1)$ depend on the initial value $\hat{z}$ and the strong concavity property of $U(\cdot)$. The above property is important for establishing the existence of the solution to Poisson equation in \Cref{ass:poisson}.  

\section{Convergence Analysis with Non-convex Loss Function} \label{app:noncvx}
We first verify the inequality \eqref{eq:biased} by observing the following expression for the gradient of performative loss:
\beq \label{eq:vgrd}
\grd V(\theta) = \grd \int_{\sf Z} \ell( \theta; z) p_{ {\cal D}(\theta) }(z) {\rm d}\!~z = \EE_{ z \sim {\cal D}(\theta) } [ \grd \ell(\theta; z) ] + \EE_{z \sim {\cal D}(\theta)} \big[ \ell(\theta;z) \grd_\theta \log( p_{ {\cal D}(\theta) }( z) ) \big],
\eeq
where we have denoted $p_{ {\cal D}(\theta) }(z)$ as the probability distribution function for ${\cal D}(\theta)$. The above identity is derived using chain rule and the property $\grd_\theta \log p_{ {\cal D}(\theta) } (z) = \frac{ \grd_\theta p_{ {\cal D}(\theta) } (z) }{ p_{ {\cal D}(\theta) } (z) }$ similar to the policy gradient theorem; see \citep[Ch.~13]{sutton2018reinforcement}.

Observe that
\beq
\pscal{ \grd V( \theta ) }{ h (\theta) } = \| h(\theta) \|^2 + \pscal{ \EE_{z \sim {\cal D}(\theta)} \big[ \ell(\theta;z) \grd_\theta \log( p_{ {\cal D}(\theta) }( z) ) \big] }{ h(\theta) }
\eeq
We note 
\beq \notag
\begin{split}
& | \pscal{ \EE_{z \sim {\cal D}(\theta)} \big[ \ell(\theta;z) \grd_\theta \log( p_{ {\cal D}(\theta) }( z) ) \big] }{ h(\theta) }| \leq \frac{1}{2} \| h(\theta)\|^2 + \frac{1}{2} \| \EE_{z \sim {\cal D}(\theta)} \big[ \ell(\theta;z) \grd_\theta \log( p_{ {\cal D}(\theta) }( z) ) \big] \|^2 \\
& \leq \frac{1}{2} \| h(\theta)\|^2 + \frac{1}{2} \EE_{z \sim {\cal D}(\theta)} \big[ |\ell(\theta;z)|^2 \| \grd_\theta \log( p_{ {\cal D}(\theta) }( z) ) \|^2 \big] \leq \frac{1}{2} \| h(\theta) \|^2 + {\rm c}_0
\end{split}
\eeq
where we have used the Jensen's inequality and set 
\beq\label{app:c_0}
{\rm c}_0 := \sup_{ \theta \in \RR^d } \frac{1}{2} \EE_{z \sim {\cal D}(\theta)} \big[ |\ell(\theta;z)|^2 \| \grd_\theta \log( p_{ {\cal D}(\theta) }( z) ) \|^2 \big].
\eeq
The above can be shown to be bounded when the loss function is bounded (e.g., a sigmoid loss), and the state-dependent distribution has bounded gradient w.r.t.~$\theta$ (e.g., a soft-max distribution). 
Together, we obtain the desired inequality:
\beq \label{eq:a1}
\pscal{ \grd V( \theta ) }{ h (\theta) } \geq \frac{1}{2} \| h(\theta) \|^2 - {\rm c}_0,~\forall~\theta \in \RR^d.
\eeq
Notice that \eqref{eq:vgrd} also implies
\beq \label{eq:a2}
\begin{split}
\| \grd V( \theta ) \| & \leq \| h( \theta ) \| + \| \EE_{z \sim {\cal D}(\theta)} \big[ \ell(\theta;z) \grd_\theta \log( p_{ {\cal D}(\theta) }( z) ) \big] \| \\
& \leq \| h( \theta ) \| + \EE_{z \sim {\cal D}(\theta)} \big[ |\ell(\theta;z)| \| \grd_\theta \log( p_{ {\cal D}(\theta) }( z) ) \| \big] \leq \| h( \theta ) \| + \sqrt{2 {\rm c}_0}.
\end{split}
\eeq

\paragraph{Proof of Corollary~\ref{cor:ncvx}} Notice that \eqref{eq:a1}, \eqref{eq:a2} imply A1, A2 of \citep{karimi2019non}, respectively. Moreover, the stated assumptions in the corollary imply A3, A5-A7 of \citep{karimi2019non}. Applying Theorem~2 from \citep{karimi2019non} shows that
\beq
\EE[ \| \grd V( \theta_{ \sf K } ) \|^2 ] \lesssim \EE[ \| h( \theta_{ \sf K } ) \|^2 ] + {\rm c}_0 \lesssim \frac{1 + \sum_{k=1}^K \gamma_k^2 }{ \sum_{k=1}^K \gamma_k } + {\rm c}_0,
\eeq
where we have omitted the constants from \citep{karimi2019non}. Note that ${\sf K} \in \{1,...,K\}$ is a discrete r.v.~selected independently with the probability $\PP ( {\sf K} = k ) = \gamma_k / \sum_{j=1}^K \gamma_j$.
Setting the step sizes as $\gamma_k = {\cal O}(1/\sqrt{k})$ shows the desired bound in the corollary.

\section{Missing Proofs in Section~\ref{sec:pf}}\label{app:miss}
Below, we present the detailed proof for the lemmas presented in \S\ref{sec:pf}.

\subsection{Proof of Lemma~\ref{lem:square}} \label{app:square}
We begin our analysis by observing that as $\grd f(\theta_{PS}; \theta_{PS} ) = 0$, we have:
\begin{align*}
& \left\|\theta_{k+1}-\theta_{PS}\right\|^{2} = \left\|\theta_{k}-\gamma_{k+1} \nabla \ell( \theta_{k}; z_{k+1} )-\theta_{PS}\right\|^{2} \\
& = \underbrace{ \left\|\theta_{k}-\theta_{PS}\right\|^{2}}_{ =: B_1 } - \, 2 \gamma_{k+1} \underbrace{ \pscal{ \theta_k -\theta_{PS} }{\nabla \ell( \theta_{k}; z_{k+1} ) - \nabla f(\theta_{PS}; \theta_{PS}) } }_{=:B_2} \\
& \qquad + \gamma_{k+1}^2 \underbrace{ \left\|  \nabla f(\theta_{PS}; \theta_{PS})- \nabla \ell ( \theta_{k}; z_{k+1} ) \right\|^{2} }_{ =: B_3 }
\end{align*}
The inner product can be lower bounded as
\begin{eqnarray}
\begin{aligned}
	B_{2}&= \pscal{ \theta_k -\theta_{PS}}{ \nabla \ell ( \theta_{k}; z_{k+1} ) -\nabla f(\theta_{PS}; \theta_{PS}) }\\
	&= \pscal{  \theta_k -\theta_{PS} } { \nabla \ell ( \theta_{k}; z_{k+1} ) -\nabla f(\theta_{k}; \theta_{k}) } + \pscal { \theta_k -\theta_{PS} }{\nabla  f(\theta_{k};\theta_{k})-\nabla f(\theta_{k}; \theta_{PS}) } \\
	&\quad + \pscal{ \theta_k -\theta_{PS} }{ \nabla f(\theta_{k};\theta_{PS})-\nabla f(\theta_{PS}; \theta_{PS}) }\\
	&\overset{(a)}{\geq} \pscal{ \theta_k -\theta_{PS}} {\nabla \ell( \theta_{k}; z_{k+1} )-\nabla f(\theta_{k}; \theta_{k}) } \\
	& \quad - \left\|\theta_{k}-\theta_{PS} \right\| \left\|\nabla f(\theta_{k}; \theta_{k})- \nabla f(\theta_{k}; \theta_{PS}) \right\|+ \mu \left\|\theta_{k}-\theta_{PS} \right\|^{2}\\
	&\overset{(b)}{\geq} \left<\theta_k -\theta_{PS}|\nabla \ell( \theta_{k}; z_{k+1} ) -\nabla f(\theta_{k}; \theta_{k})\right> +(\mu-L\varepsilon) \left\|\theta_{k}-\theta_{PS} \right\|^{2}
\end{aligned}
\end{eqnarray}
where (a) is due to the Cauchy-schwarz inequality and the $\mu$-strong convexity of $\nabla f(\cdot; \cdot)$; (b) is due to the $L$-smoothness of $f$ and the $\epsilon$-sensitivity of the distribution [c.f~Assumption~\ref{ass:sensitive}]; also see \cite{perdomo2020performative}. Furthermore, 
\begin{eqnarray}
\begin{aligned}
B_{3}&= \left\|  \nabla \ell ( \theta_{k}; z_{k+1} )-\nabla f(\theta_{PS}; \theta_{PS})+\nabla \ell ( \theta_{PS}; z_{k+1} ) -\nabla \ell ( \theta_{PS}; z_{k+1} ) \right\|^{2} \\
   &\leq 2 \left( \left\| \nabla\ell ( \theta_{PS}; z_{k+1} )-\nabla\ell ( \theta_{k}; z_{k+1} ) \right\|^2+\left\|\nabla f(\theta_{PS}; \theta_{PS})-\nabla\ell ( \theta_{PS}; z_{k+1} ) \right\|^2  \right)\\
   &\leq 2L^2 \left\|\theta_{k}-\theta_{PS} \right\|^2 + 2 \sigma^2 
\end{aligned}
\end{eqnarray}
where the third inequality is due to Assumptions \ref{ass:gradient}, \ref{ass:bounded}. 
Combing the bounds for $B_1$, $B_{2}$ and $B_3$, we can get the desired inequality.
\begin{eqnarray}\label{eq:recursion}
\begin{aligned}
	& \left\|\theta_{k+1}-\theta_{PS}\right\|^2 \\
	&\leq \left\|\theta_{k}-\theta_{PS}\right\|^2+ 2 \gamma_{k+1}^2\cdot \left(\sigma^2+L^2 \left\|\theta_{k}-\theta_{PS}\right\|^2\right)\\
	&\quad -2\gamma_{k+1}\left(\left<\theta_k -\theta_{PS}|\nabla \ell( \theta_{k}; z_{k+1} )-\nabla f(\theta_{k}; \theta_{k})\right> +(\mu-L\varepsilon) \left\|\theta_{k}-\theta_{PS} \right\|^{2}\right) \\
	& = \big( 1 - 2 \gamma_{k+1} ( \mu - L \epsilon ) + 2 \gamma_{k+1}^2 L^2 \big) \| \theta_k - \theta_{PS} \|^2 \\
	& \quad  + 2 \gamma_{k+1}^2 \sigma^2 - 2\gamma_{k+1} \pscal{ \theta_k - \theta_{PS} }{ \grd \ell( \theta_k; z_{k+1} ) - \grd f( \theta_k; \theta_k ) }.
\end{aligned}
\end{eqnarray}

It is noted that if we consider a case when the SA scheme \eqref{eq:mkv_setting} is non-state-dependent, e.g., $z_{k+1}$ is drawn from $\mathcal{D}(\theta_{k})$  independently, then proving Lemma \ref{lem:square} suffices to show our desired Theorem \ref{th:main} since the last term in equation (\ref{eq:square_1st}) is zero mean when conditioned on the previous iterates [cf. \eqref{eq:inter}]. 

\subsection{Proof of Lemma~\ref{lem:poisson}} \label{app:poisson}
Applying Assumption~\ref{ass:poisson} shows that the sum of inner product can be evaluated as 
\begin{eqnarray*}
\begin{aligned}
	& \sum_{s=1}^{k}\gamma_{s}G_{s+1:k}\EE\pscal{ \theta_{PS} - \theta_{s-1} } {\nabla \ell( \theta_{s-1}; z_s )-\nabla f(\theta_{s-1}; \theta_{s-1}) }  \\
	&= \sum_{s=1}^{k}\gamma_{s}G_{s+1:k}\EE\pscal{ \theta_{PS} - \theta_{s-1} } {\widehat{\nabla\ell}( \theta_{s-1}; z_s )-\MK_{\theta_{s-1}} \widehat{\nabla\ell}( \theta_{s-1}; z_s )} \equiv \EE\left(A_1+A_2+A_3+A_4+A_5\right),
\end{aligned}
\end{eqnarray*}
where we decomposed the sum of inner product into five sub-terms $A_{1}$, $A_{2}$, $A_{3}$, $A_{4}$, $A_{5}$ such that
\begingroup
\allowdisplaybreaks
\begin{align*}
	&A_{1}:=-\sum_{s=2}^{k}\gamma_{s}G_{s+1:k}\pscal{ \theta_{s-1}-\theta_{PS}} { \widehat{\nabla\ell}( \theta_{s-1}; z_s )-\MK_{\theta_{s-1}}\widehat{\nabla \ell}( \theta_{s-1}; z_{s-1} ) } \\
	&A_{2}:=-\sum_{s=2}^{k}\gamma_{s}G_{s+1:k} \pscal{ \theta_{s-1}-\theta_{PS} }{ \MK_{\theta_{s-1}} \widehat{\nabla\ell}( \theta_{s-1}; z_{s-1} )-\MK_{\theta_{s-2}}\widehat{\nabla\ell}( \theta_{s-2}; z_{s-1} ) }\\
	&A_{3}:=-\sum_{s=2}^{k}\gamma_{s}G_{s+1:k} \pscal{ \theta_{s-1}-\theta_{s-2} } { \MK_{\theta_{s-2}}\widehat{\nabla\ell}( \theta_{s-2}; z_{s-1} ) }\\
	&A_{4}:=-\sum_{s=2}^{k}(\gamma_{s}G_{s+1:k}-\gamma_{s-1}G_{s:k})\pscal{ \theta_{s-2}-\theta_{PS} }{ \MK_{\theta_{s-2}}\widehat{\nabla\ell}(\theta_{s-2}; z_{s-1}) }\\
	&A_{5}:=-\gamma_{1}G_{2:k}\pscal{ \theta_{0}-\theta_{PS} } {\widehat{\nabla\ell}( \theta_{0}; z_1 ) } +	
	\gamma_{k}  \pscal{ \theta_{k-1}-\theta_{PS} }{ \MK_{\theta_{k-1}}\widehat{\nabla\ell}( \theta_{k-1}; z_k ) } .
\end{align*}
\endgroup

We remark that a similar decomposition can be found in \citep{benveniste2012adaptive}. However, \cite{benveniste2012adaptive} proceeded with the analysis by assuming that $\theta_k$ stays in the compact set for all $k \geq 0$. We do not make such assumption in this work.

For $A_{1},$ we note that $\widehat{\nabla\ell}( \theta_{s-1}; z_s )-\MK_{\theta_{s-1}}\widehat{\nabla\ell}( \theta_{s-1}; z_{s-1} )$ is a martingale difference sequence and therefore we have $\EE\left[A_{1}\right]=0$ by taking the total expectation.

For $A_{2}$, as $\theta_{k+1}=\theta_{k}-\gamma_{k+1} \nabla {\ell}\left(  \theta_{k}; z_{k+1} \right) $, we get $\theta_{s-1}-\theta_{s-2}=-\gamma_{s-1} \nabla {\ell} ( \theta_{s-2}; z_{s-1})$. Applying the smoothness condition  \Cref{ass:poisson_bound} shows that
\beq
\begin{split}
    A_{2}&=-\sum_{s=2}^{k}\gamma_{s}G_{s+1:k}\pscal {\theta_{s-1}-\theta_{PS}} {\MK_{\theta_{s-1}} \widehat{\nabla\ell}( \theta_{s-1}; z_{s-1} )-\MK_{\theta_{s-2}}\widehat{\nabla\ell}( \theta_{s-2}; z_{s-1}) }\\
	&
	\leq  \Lph \sum_{s=2}^k \gamma_{s} G_{s+1:k} \left\|\theta_{s-1}-\theta_{PS}\right\| \left\|\theta_{s-1}-\theta_{s-2}\right\|
	\\
	&\leq \Lph \sum_{s=2}^k  \gamma_{s-1} \gamma_{s} G_{s+1:k} \left\|\theta_{s-1}-\theta_{PS}\right\|  \left\|\grd \ell( \theta_{s-2}; z_{s-1} )\right\|.
\end{split}
\eeq
Combining with the implied bound \eqref{eq:corr} from the assumptions as well as \eqref{eq:stepsize_cond} yield
\begin{align*}
	A_2  
	&
	\leq \varsigma \Lph \overline{L} \, \sum_{s=2}^k \gamma_{s}^2  G_{s+1:k} \|\theta_{s-1}-\theta_{PS} \|  \left( 1 + \| \theta_{s-2} - \theta_{PS} \| \right) 
	\\
	&
	\leq \varsigma \Lph \overline{L} \, \sum_{s=2}^k \gamma_{s}^2 G_{s+1:k}  \left\{ \frac{1}{2} + \frac{1}{2} \| \theta_{s-2} - \theta_{PS} \|^2 +  \| \theta_{s-1} - \theta_{PS} \|^2 
	\right\}
	\\
	& 
	\leq  \varsigma \Lph \overline{L} \, \Big\{ \frac{1}{2} \sum_{s=2}^k \gamma_s^2 G_{s+1:k} + \frac{1}{2} \sum_{s=2}^k \gamma_s^2 G_{s+1:k} \| \theta_{s-2} - \theta_{PS} \|^2 +  \sum_{s=2}^k \gamma_s^2 G_{s+1:k} \| \theta_{s-1} - \theta_{PS} \|^2 \Big\} ,
\end{align*}
where the second inequality applies $a(1 + c) \leq \frac{1}{2} + \frac{1}{2} c^2 +  a^2 $ for any $a,c\in \RR$. 

For $A_{3}$, again using \eqref{eq:corr}, we observe that
\beq
\begin{split}
    A_{3}&=-\sum_{s=2}^{k}\gamma_{s}G_{s+1:k}\pscal{ \theta_{s-1}-\theta_{s-2} }{ \MK_{\theta_{s-2}}\widehat{\nabla\ell}( \theta_{s-2}; z_{s-1} ) }\\	
	&\leq \sum_{s=2}^{k}\gamma_{s} G_{s+1:k}\left\|\theta_{s-1}-\theta_{s-2}\right\|\cdot \left\|\MK_{\theta_{s-2}}\widehat{\nabla\ell}( \theta_{s-2}; z_{s-1} )\right\| \\
	&\leq \sum_{s=2}^{k}\gamma_{s}\gamma_{s-1} G_{s+1:k} \norm{\grd \ell( \theta_{s-2}; z_{s-1} )}\cdot \LZ \left( 1 + \norm{\theta_{s-2} - \theta_{PS} } \right)\\
	&\leq \varsigma \overline{L} \LZ \sum_{s=2}^{k}\gamma_{s}^2 G_{s+1:k} (1 + \| \theta_{s-2} - \theta_{PS} \| )^2 \\
	&\leq 2 \varsigma \overline{L} \LZ \sum_{s=2}^{k}\gamma_{s}^2 G_{s+1:k} \{ 1 + \| \theta_{s-2} - \theta_{PS} \|^2 \} .
\end{split}
\eeq

For $A_{4}$, we notice that
\beq
\begin{split}
    A_{4}&=-\sum_{s=2}^{k}(\gamma_{s}G_{s+1:k}-\gamma_{s-1}G_{s:k}) \pscal{ \theta_{s-2}-\theta_{PS}} { \MK_{\theta_{s-2}}\widehat{\nabla\ell}( \theta_{s-2}; z_{s-1}) } \\
	&\leq \sum_{s=2}^{k} | \gamma_{s}G_{s+1:k}-\gamma_{s-1}G_{s:k} | \left\|\theta_{s-2}-\theta_{PS} \right\|\cdot \left\|\MK_{\theta_{s-2}}\widehat{\grd\ell}(\theta_{s-2}; z_{s-1} )\right\| .
\end{split}
\eeq
It can be shown that $| \gamma_{s}G_{s+1:k}-\gamma_{s-1}G_{s:k} | \leq ( 1 + \tilde{\mu} ) \varsigma \gamma_s^2 G_{s+1:k} $, therefore 
\beq
\begin{split}
	A_4 &\leq ( 1 + \tilde{\mu} ) \varsigma \LZ \sum_{s=2}^{k} \gamma_{s}^{2} G_{s+1:k} \left\|\theta_{s-2}-\theta_{PS} \right\| \left( 1 + \norm{ \theta_{s-2} - \theta_{PS} } \right)\\
	& \leq ( 1 + \tilde{\mu} ) \varsigma \LZ \sum_{s=2}^{k} \gamma_{s}^{2} G_{s+1:k} \Big\{ \frac{1}{2} + \frac{3}{2} \| \theta_{s-2} - \theta_{PS} \|^2 \Big\} \\
	& \leq ( 1 + \tilde{\mu} ) \varsigma \LZ \Big\{ \frac{1}{2} \sum_{s=2}^{k} \gamma_{s}^{2} G_{s+1:k} + \frac{3}{2} \sum_{s=2}^{k} \gamma_{s}^{2} G_{s+1:k} \| \theta_{s-2} - \theta_{PS} \|^2 \Big\}.
\end{split}
\eeq

Finally, for $A_{5}$, we have 
\begin{align*}
	A_{5}&=-\gamma_{1}G_{2:k}\pscal{ \theta_{0}-\theta_{PS}} { \widehat{\nabla\ell}(\theta_{0}; z_{1} ) } +	
	\gamma_{k} \pscal{ \theta_{k-1}-\theta_{PS}} {\MK_{\theta_{k-1}}\widehat{\nabla\ell} ( \theta_{k-1}; z_{k} ) }\\
	&\leq \gamma_{1} G_{2:k} \left\|\theta_{0}-\theta_{PS} \right\| \left\|\widehat{\grd\ell}( \theta_{0}; z_{1} )\right\|+	
	\gamma_{k} \left\|\theta_{k-1}-\theta_{PS}\right\|\left\|\MK_{\theta_{k-1}}\widehat{\grd\ell}( \theta_{k-1}; z_{k}) \right\|  \\
	& \leq \gamma_1 \widehat{L} \, G_{2:k} \| \theta_0 - \theta_{PS} \| \big( 1 + \| \theta_0 - \theta_{PS} \| \big) + \gamma_k \LZ \, \| \theta_{k-1} - \theta_{PS} \| \big( 1 + \| \theta_{k-1} - \theta_{PS} \| \big) \\
	& \leq \frac{\gamma_1 \widehat{L} \, G_{2:k}}{2} + \frac{ \gamma_k \LZ }{2} + \frac{3 \gamma_1 \widehat{L}}{2} \, G_{2:k} \| \theta_0 - \theta_{PS} \|^2 + \frac{3 \gamma_k \LZ}{2} \, \| \theta_{k-1} - \theta_{PS} \|^2
\end{align*}
Summing up $A_{1}$ to $A_{5}$ and taking the full expectation yield:
\begin{align*} 
	& 2 \big| \EE\left[A_1+A_2+A_3+A_4+A_5\right] \big| \\
	& \leq \varsigma \Lph \overline{L} \, \Big\{ \sum_{s=2}^k \gamma_s^2 G_{s+1:k} + \sum_{s=2}^k \gamma_s^2 G_{s+1:k} \Delta_{s-2} + 2\sum_{s=2}^k \gamma_s^2 G_{s+1:k} \Delta_{s-1}  \Big\} \\
	& \quad + 4 \varsigma \overline{L} \LZ \sum_{s=2}^{k}\gamma_{s}^2 G_{s+1:k} \big\{ 1 + \Delta_{s-2} \big\} + ( 1 + \tilde{\mu} ) \varsigma \LZ \Big\{ \sum_{s=2}^{k} \gamma_{s}^{2} G_{s+1:k} + 3 \sum_{s=2}^{k} \gamma_{s}^{2} G_{s+1:k} \Delta_{s-2} \Big\} \\
	&\quad + \gamma_1 \widehat{L} \, G_{2:k} + \gamma_k \LZ  + 3 \gamma_1 \widehat{L} \, G_{2:k} \Delta_0 + 3 \gamma_k \LZ \, \Delta_{k-1} .
\end{align*}
Recall the following constants:
\beq \label{cone_ctwo}
	\Cone := \varsigma \Lph \overline{L} + 4 \varsigma \overline{L} \LZ + ( 1 + \tilde{\mu} ) \varsigma \LZ, ~~ \Ctwo := 2 \varsigma \Lph \overline{L}, ~~ \Cthree := \varsigma \Lph \overline{L} + 4 \varsigma \overline{L} \LZ + 3 ( 1 + \tilde{\mu} )\varsigma \LZ.
\eeq
We obtain the desirable bound for the lemma:
\begin{align*} 
	& 2 \big| \EE\left[A_1+A_2+A_3+A_4+A_5\right] \big| \\
	& \leq \sum_{s=2}^k \gamma_s^2 G_{s+1:k} \big( \Cone + \Ctwo \Delta_{s-1} + \Cthree \Delta_{s-2} \big) + \LZ \gamma_k \big\{ 1 + 3 \Delta_{k-1} \big\} + \gamma_1 G_{2:k} \big( \LZ (1 + 3 \Delta_0 ) + \gamma_1 \Cone \big).
\end{align*}
This concludes the proof.

\subsection{Proof of Lemma~\ref{lem:bdd}} \label{app:bdd}
Consider the inequality in \eqref{eq:lem5}. We consider a non-negative upper bound sequence $\{ {\rm U}_k \}_{k \geq 0}$ defined by the recursion:
\beq 
\begin{split}
	{\rm U}_k & = G_{1:k} {\rm U}_0 + \Big( \frac{ 2}{ \tilde{\mu} } (2 \sigma^2 + \Cone) + \LZ \Big) \gamma_k + \sum_{s=1}^{k-1} \gamma_{s+1}^2 G_{s+2:k} \big( \Ctwo {\rm U}_s + \Cthree {\rm U}_{s-1} \big) \\
	& \quad + \gamma_1 G_{2:k} \big\{ \LZ (1 + 3 {\rm U}_0 ) + \gamma_1 ( 2 \sigma^2 + \Cone ) \big\} + 3 \gamma_k \LZ {\rm U}_{k-1} , \\[.2cm]
\end{split} 
\eeq
for any $k \geq 1$, and we have defined ${\rm U}_0 = \Delta_0$.
Notice that by construction, we have $\Delta_k \leq {\rm U}_k$ for any $k \geq 0$. 

Using the convention that ${\rm U}_{-1} = 0$, we observe that for any $k \geq 1$, 
\beq \label{eq:further_upd}
\begin{split}
{\rm U}_k & = ( 1 - \gamma_k \tilde{\mu} ) {\rm U}_{k-1} + \Big( \frac{ 2}{ \tilde{\mu} } (2 \sigma^2 + \Cone) + \LZ \Big) \Big( \gamma_k - (1 - \gamma_k \tilde{\mu} ) \gamma_{k-1} \Big) \\
& \quad + \gamma_{k}^2 ( \Ctwo {\rm U}_{k-1} + \Cthree {\rm U}_{k-2} ) + 3 \LZ \Big( \gamma_k {\rm U}_{k-1} - (1 - \gamma_k \tilde{\mu}) \gamma_{k-1} {\rm U}_{k-2} \Big) \\
& \leq \big( 1 - \gamma_k \tilde{\mu} + \gamma_{k}^2 \Ctwo \big) {\rm U}_{k-1} + \Big( \frac{ 2}{ \tilde{\mu} } (2 \sigma^2 + \Cone) + \LZ \Big) \tilde{\mu} \varsigma \gamma_k^2 + \Big( \Cthree + 3 \LZ \tilde{\mu} \Big)\gamma_{k} \gamma_{k-1} {\rm U}_{k-2} \\
& \quad + 3 \LZ \big( \gamma_k {\rm U}_{k-1} - \gamma_{k-1} {\rm U}_{k-2} \big) \\
& \leq \big( 1 - \gamma_k \tilde{\mu} / 2 \big) {\rm U}_{k-1} + \Big( \frac{ 2}{ \tilde{\mu} } (2 \sigma^2 + \Cone) + \LZ \Big) \tilde{\mu} \varsigma \gamma_k^2 + \Big( \Cthree + 3 \LZ \tilde{\mu} \Big)\gamma_{k} \gamma_{k-1} {\rm U}_{k-2} \\
& \quad + 3 \LZ \big( \gamma_k {\rm U}_{k-1} - \gamma_{k-1} {\rm U}_{k-2} \big),
\end{split}
\eeq
where the last inequality is due to $\gamma_k \leq \tilde{\mu}/2 \Ctwo$.

We prove part {\sf (i)} of the lemma. 
From \eqref{eq:further_upd}, we consider an upper bound sequence $\{ \oU_k \}_{k \geq -1}$ defined by the recursion:
\beq \label{eq:tobesum}
\begin{split}
\oU_k & = \big( 1 - \gamma_k \tilde{\mu} / 2 \big) \oU_{k-1} + \Big( \frac{ 2}{ \tilde{\mu} } (2 \sigma^2 + \Cone) + \LZ \Big) \tilde{\mu} \varsigma \gamma_k^2 + \Big( \Cthree + 3 \LZ \tilde{\mu} \Big)\gamma_{k} \gamma_{k-1} \oU_{k-2} \\
& \quad + 3 \LZ \big( \gamma_k \oU_{k-1} - \gamma_{k-1} \oU_{k-2} \big), \quad \forall k \geq 1.
\end{split}
\eeq
We have also defined $\oU_0 = {\rm U}_0$, $\oU_{-1} = 0$. 
For any $t \geq 1$, summing up the equation \eqref{eq:tobesum} from $k=1$ to $k=t$ yields
\beq \notag
\begin{split}
\sum_{k=1}^t \oU_k & = \sum_{k=1}^t \Big\{ \big( 1 - \gamma_k \tilde{\mu} / 2 \big) \oU_{k-1} + \Big( \frac{ 2}{ \tilde{\mu} } (2 \sigma^2 + \Cone) + \LZ \Big) \tilde{\mu} \varsigma \gamma_k^2 + \Big( \Cthree + 3 \LZ \tilde{\mu} \Big)\gamma_{k} \gamma_{k-1} \oU_{k-2} \Big\} \\
& \quad + 3 \LZ \sum_{k=1}^t \big( \gamma_k \oU_{k-1} - \gamma_{k-1} \oU_{k-2} \big),
\end{split}
\eeq
Rearranging terms leads to 
\beq \notag
\oU_t = \oU_0 +  \sum_{k=1}^t \Big\{  \Big( \Cthree + 3 \LZ \tilde{\mu} \Big) \gamma_k \gamma_{k-1} \oU_{k-2} + \Big( \frac{ 2}{ \tilde{\mu} } (2 \sigma^2 + \Cone) + \LZ \Big) \tilde{\mu} \varsigma \gamma_k^2 -  \frac{\tilde{\mu}}{2} \gamma_k \oU_{k-1} \Big\} + 3 \LZ  \gamma_t \oU_{t-1} 
\eeq
Using the step size conditions $\gamma_k \leq \gamma_{k-1}$, $\gamma_k \leq \big( \Cthree + 3 \LZ \tilde{\mu} \big)^{-1} \min\{  \tilde{\mu} / 2 , 3 \LZ \}$, $\gamma_k \leq (6 \LZ)^{-1}$,  
\beq
\begin{split}
\oU_t & \leq \oU_0 + 3 \LZ \gamma_t \oU_{t-1}  \\
& \quad  + \sum_{k=1}^t \Big\{ \Big[ \Big( \Cthree + 3 \LZ \tilde{\mu} \Big) \gamma_k^2 - \frac{\tilde{\mu}}{2} \gamma_k \Big] \oU_{k-1} + \Big( \frac{ 2}{ \tilde{\mu} } (2 \sigma^2 + \Cone) + \LZ \Big) \tilde{\mu} \varsigma \gamma_k^2 \Big\} \\
& \leq 3 \LZ \gamma_t \oU_{t-1} + \oU_0 + \Big( \frac{ 2}{ \tilde{\mu} } (2 \sigma^2 + \Cone) + \LZ \Big) \tilde{\mu}  \varsigma \sum_{k=1}^t \gamma_k^2 \\
& \leq \frac{1}{2} \oU_{t-1} + \Big\{ \oU_0 + \Big( \frac{ 2}{ \tilde{\mu} } (2 \sigma^2 + \Cone) + \LZ \Big) \tilde{\mu} \varsigma \sum_{k=1}^t \gamma_k^2 \Big\},
\end{split}
\eeq
where we obtain the first inequality after  shifting the summation's index and it is noted that $\oU_{-1}=0.$ Rearranging terms and solving the recursion lead to 
\beq
\begin{split}
\oU_t & \leq \big( \frac{1}{2} \big)^{t} \oU_0 + \sum_{s=1}^t \big( \frac{1}{2} \big)^{t-s} \Big\{ \oU_0 + \Big( \frac{ 2}{ \tilde{\mu} } (2 \sigma^2 + \Cone) + \LZ \Big) \tilde{\mu} \varsigma \sum_{k=1}^s \gamma_k^2 \Big\} \\
& \leq 3 \oU_0 + 2 \tilde{\mu} \varsigma \Big( \frac{ 2}{ \tilde{\mu} } (2 \sigma^2 + \Cone) + \LZ \Big) \sum_{\ell=1}^t \gamma_\ell^2 \big( \frac{1}{2} \big)^{t-\ell} \\
&\leq 3 \oU_0 + \frac{ \tilde{\mu} \varsigma}{ 9 \LZ^2 } \Big( \frac{ 2}{ \tilde{\mu} } (2 \sigma^2 + \Cone) + \LZ \Big)
\end{split}
\eeq
Recall that $\overline{\Delta} \eqdef 3 \oU_0 + \frac{ \varsigma}{ 9 \LZ^2 } \Big( 2 (2 \sigma^2 + \Cone) +  \tilde{\mu} \LZ \Big)$, the above shows $\Delta_t \leq {\rm U}_t \leq \oU_t \leq \overline{\Delta}$ for any $t \geq 1$, thus establishing part {\sf (i)}.

We now proceed to proving part {\sf (ii)} of the lemma. We define $\oG_{m:n} = \prod_{\ell=m}^n (1 - \gamma_\ell \tilde{\mu}/2)$ and observe from \eqref{eq:tobesum} that
\beq
\begin{split}
{\rm U}_k & \leq \oG_{1:k} {\rm U}_0 + \sum_{s=1}^k \oG_{s+1:k} \Big\{ \Big( \frac{ 2}{ \tilde{\mu} } (2 \sigma^2 + \Cone) + \LZ \Big) \tilde{\mu} \varsigma \gamma_s^2 + \Big( \Cthree + 3 \LZ \tilde{\mu} \Big)\gamma_{s} \gamma_{s-1} {\rm U}_{s-2} \Big\} \\
& \quad + 3 \LZ  \sum_{s=1}^k \oG_{s+1:k} \Big\{  \big( \gamma_s {\rm U}_{s-1} - \gamma_{s-1} {\rm U}_{s-2} \big) \Big\}.
\end{split}
\eeq
Notice that 
\beq
\begin{split}
	& \sum_{s=1}^k \oG_{s+1:k} \big( \gamma_s {\rm U}_{s-1} - \gamma_{s-1} {\rm U}_{s-2} \big) \\
	& = \sum_{s=1}^k \oG_{s+1:k} \big( \gamma_s {\rm U}_{s-1} + (1 - \gamma_s \tilde{\mu} / 2) ( \gamma_{s-1} {\rm U}_{s-2} - \gamma_{s-1} {\rm U}_{s-2} ) - \gamma_{s-1} {\rm U}_{s-2} \big) \\
	& = \sum_{s=1}^k \Big\{ \Big( \oG_{s+1:k} \gamma_s {\rm U}_{s-1} - \oG_{s:k} \gamma_{s-1} {\rm U}_{s-2} \Big) - \gamma_s \gamma_{s-1} \tilde{\mu} {\rm U}_{s-2} / 2 \Big\} \leq \gamma_k {\rm U}_{k-1} \leq \gamma_k \overline{\Delta}.
\end{split} 
\eeq
By Lemma~\ref{lem:o_gamma_k}, we have $\sum_{s=1}^k \oG_{s+1:k} \gamma_s^2 \leq 4 \gamma_k / \tilde{\mu}$ and the following is obtained
\beq
{\rm U}_k \leq \oG_{1:k} {\rm U}_0 + \Big\{ \frac{4 \varsigma}{\tilde{\mu}} \Big( 2 (2 \sigma^2 + \Cone) + \tilde{\mu} \LZ \Big) + 3 \LZ \overline{\Delta}  + \frac{4 \varsigma}{\tilde{\mu}} \Big( \Cthree + 3 \LZ \tilde{\mu} \Big) \overline{\Delta} \Big\} \gamma_k .
\eeq
The proof is completed.

\subsection{Auxiliary Lemmas}

\begin{Lemma}\label{lem:o_gamma_k}
	Let $a>0$ and $\left(\gamma_{k}\right)_{k \geq 1}$ be a non-increasing sequence such that $\gamma_{1}< 2 / a$. If $\gamma_{k-1} / \gamma_k \leq 1 + (a/2)\gamma_k$ for any $k \geq 1$, then for any $k \geq 2$, 
\beq
\sum_{j=1}^{k} \gamma_{j}^{2} \prod_{\ell=j+1}^{k}\left(1-\gamma_{\ell} a\right) \leq  \frac{2}{a} \gamma_{k}.
\eeq
\end{Lemma}

\begin{proof}
The proof is elementary. Observe that:
\beq 
\begin{split}
   \sum_{j=1}^{k} \gamma_{j}^{2} \prod_{\ell=j+1}^{k}\left(1-\gamma_{\ell} a\right)  & = \gamma_k \sum_{j=1}^{k} \gamma_{j} \prod_{\ell=j+1}^{k} \frac{ \gamma_{\ell-1} }{ \gamma_\ell } \left(1-\gamma_{\ell} a\right) \\
   & \leq \gamma_k \sum_{j=1}^{k} \gamma_{j} \prod_{\ell=j+1}^{k} (1 + (a/2) \gamma_\ell ) \left(1-\gamma_{\ell} a\right) \\
   & \leq \gamma_k \sum_{j=1}^{k} \gamma_{j} \prod_{\ell=j+1}^{k} \left(1-\gamma_{\ell} (a/2) \right) \\
  & = \frac{2 \gamma_k}{a} \sum_{j=1}^{k} \left( \prod_{\ell=j+1}^{k} (1 - \gamma_\ell a/2 ) - \prod_{\ell'=j}^{k} (1 - \gamma_{\ell'} a/2 ) \right) \\
  & = \frac{2 \gamma_k}{a} \left( 1 - \prod_{\ell'=1}^{k} (1 - \gamma_{\ell'} a/2 ) \right) \leq \frac{2 \gamma_k}{a}.
\end{split}
\eeq
The proof is concluded.
\end{proof}

\section{Details of the Numerical Experiments} \label{app:details}
\vspace{-0.2cm}
    
    
This section provides details about the numerical experiments on the second problem of strategic classification (SC) in \S\ref{sec:num}. Moreover, we provide additional experiment results to better illustrate the performance of the state dependent SA algorithm for this problem.

The experiments conducted in this section are based on the Credit simulator provided at \url{https://github.com/zykls/performative-prediction}.
Our experiments are conducted on a server with Intel Xeon Gold 6138 CPU. The Python codes are executed in a single-thread environment. 

There are two roles in the SC problem -- \emph{learner} and \emph{agents}. The learner utilizes agents' information to obtain a classifier $f_{\theta}$. Meanwhile, individual agents hope to be assigned to a favorable class. To do so, they modify their features and thereby shifting the data distribution towards the target ${\cal D}(\theta)$. 
Specifically, our experiments are done on the \texttt{GiveMeSomeCredit} dataset with $m=18357$ samples as we select $d=3$ features to build the classifier. Each (original) data sample is given by $\bar{z}_i = (\bar{x}_i, \bar{y}_i)$ with the label $\bar{y}_i \in \{0,1\}$ and selected feature $\bar{x}_i \in \RR^3$. We associate each data sample to an agent. The task for the learner (bank) is to design a classifier that distinguishes whether the application of an individual (agent) who want to default a loan should be granted or not. 

We simulate the \emph{adapted best response} presented in \textbf{Example 1} of the main paper. In this setting, the agents rely on their past experience to present data to the learner that is favorable to to agents. The latter is achieved by a gradient descent step that depends on the current learner's state ($\theta_k$), past agent's state ($z_k$) and the original data (${\cal D}_0$).
As the dynamics is coupled between the agents' and learner's update, we present the overall algorithm based on \eqref{eq:sa}, \eqref{eq:mkv_setting} as follows:
\begin{center}
    \begin{mdframed}[innertopmargin=8pt,linecolor=black!20,linewidth=1pt]
    \begin{center}
        \textbf{\underline{Algorithm 2: State-dependent SA with Adapted Best Response.}}
    \end{center}
    \begin{enumerate}[leftmargin=11mm]
    \item[\textbf{Input}:] initial iterate $\theta_0 \in \RR^d$, agents' state $x_i^0 = \bar{x}_i$, $i \in \{1,...,m\}$ such that $\bar{x}_i$ is the $i$th original feature vector, step sizes $\{ \gamma_k \}_{k \geq 0}$, agents' response rate $\alpha > 0$, update parameter ${\sf b}$.\vspace{-.1cm}
    \item[\textbf{For}] $k=0,1,2, \ldots$ \vspace{-.1cm}
    \item A subset of \emph{agents}, ${\cal I}_k$ with $|{\cal I}_k| = {\sf b}$, is selected uniformly from $\{1,...,m\}$. They adapt their feature vectors based on past experience and $\theta_k$ as:
    \beq \label{eq:abr}
    x_i^{k+1} = x_i^k + \alpha \grd U( x_i^k ; \bar{z}_i, \theta_k ),~\forall~i \in {\cal I}_k, \quad x_i^{k+1} = x_i^k,~\forall~i \notin {\cal I}_k. \vspace{-.1cm}
    \eeq
    \item An agent $i_k \in \{1,...,m\}$ is drawn uniformly to present data. Set $z_{k+1} = ( x_{i_k}^{k+1} , y_{i_k})$.\vspace{-.1cm}
    \item The \emph{learner} computes the $k+1$th iterate by:\vspace{-.1cm}
    \beq \notag
	\theta_{k+1}=\theta_k-\gamma_{k+1} \grd \ell(\theta_{k}; z_{k+1}). 
    \eeq
    The most recent iterate $\theta_{k+1}$ is deployed and made available to the agent(s). 
    \end{enumerate}
    \end{mdframed}
\end{center}
Steps 1 \& 2 in Algorithm~2 resemble the adaptive best response update in \eqref{eq:ibr}. We emphasize that these two steps are \emph{agnostic} to the learner as the latter only sees $z_{k+1}$ at iteration $k$, similarly, the last step is not known to the agents as the latter only sees the classifier given as $\theta_{k+1}$.

Furthermore, we recall that the following two types of utility functions are considered as $U(\cdot)$:
\beq 
\begin{split}
    U_{\sf q}(x^{\prime};{z}, \theta) & = \pscal{\theta}{x'}- \frac{\|x'-x\|^{2}}{2\epsilon}, \\
    U_{\sf lg}(x^{\prime};{z}, \theta) & = y \pscal{\theta}{x'} - \log\left(1+\exp(\pscal{\theta}{x'})\right)- \frac{\|x'-x\|^{2}}{2\epsilon}.
\end{split}
\eeq
In step 1, the agents' response rate $\alpha$ and parameter ${\sf b}$ control the speed of adaptation among the group of $m$ agents. These parameters will affect the mixing time of the MC which determines the bounds in \Cref{th:main}. 
Overall, we observe that the agents' states and learner's iterates are evolving simultaneously, highlighting the coupled nature in the analysis of the state-dependent SA algorithm. 

In cases such as $U_{\sf lg}(\cdot)$ where the ideal best response $\argmax_{x'} U(x'; z, \theta)$ must be obtained via an iterative algorithm. From an algorithmic standpoint, the stateful nature for the agent is necessary for the performative prediction algorithm to converge to $\theta_{PS}$. 

\paragraph{Additional Experiments} Next, we provide additional experiments to illustrate the performance of the state-dependent SA algorithm from a few additional perspectives. Unless otherwise specified, we adopt the same parameters set in the experiments presented in the main paper. 
In particular, we set $\beta = 1000/m$ in \eqref{eq:obj_cred}, $\epsilon = 0.01$ in the utility functions, and  in \eqref{eq:ibr}, we set number of selected agents $|{\cal I}_k| = 5$, agents' response rate $\alpha = 0.5 \epsilon$. 
The step size for \eqref{eq:sa} is $\gamma_k = {c_0} / {(c_1+k)}$, $c_0 = {100} / {\tilde{\mu}}, c_1 = { 8 L^2 } / {\tilde{\mu}^2}$, where $L, \tilde{\mu}$ are estimated as $\sqrt{2 \beta m + \| X \|_F^2 / 2}$, $(1-\epsilon) \beta - \epsilon \| X \|_F^2 / 4m$, respectively.  By default, the SA algorithm is executed as presented in Algorithm 2 with a batch size of ${\sf batch} = 1$ and the agents perform only ${\sf BR}=1$ best response update per SA update in step 3 of Algorithm~2. 

Besides, we compare the convergence rates of the algorithms from the perspective of the \emph{agents} -- measured by the number of BR updates performed by the agents. This is the setting used in the plot of Fig.~\ref{fig:main_sim} (right) and is denoted with the $x$-axis label of `{\sf no.~of agent update}'. We also compare the convergence from the perspective of the \emph{learner} -- measured by the number of samples requested from the agents by the learner. This setting is denoted with the $x$-axis label of `{\sf no.~of samples drawn by learner}'.

 \begin{figure}[h]
     \centering
     \includegraphics[width=.45\linewidth]{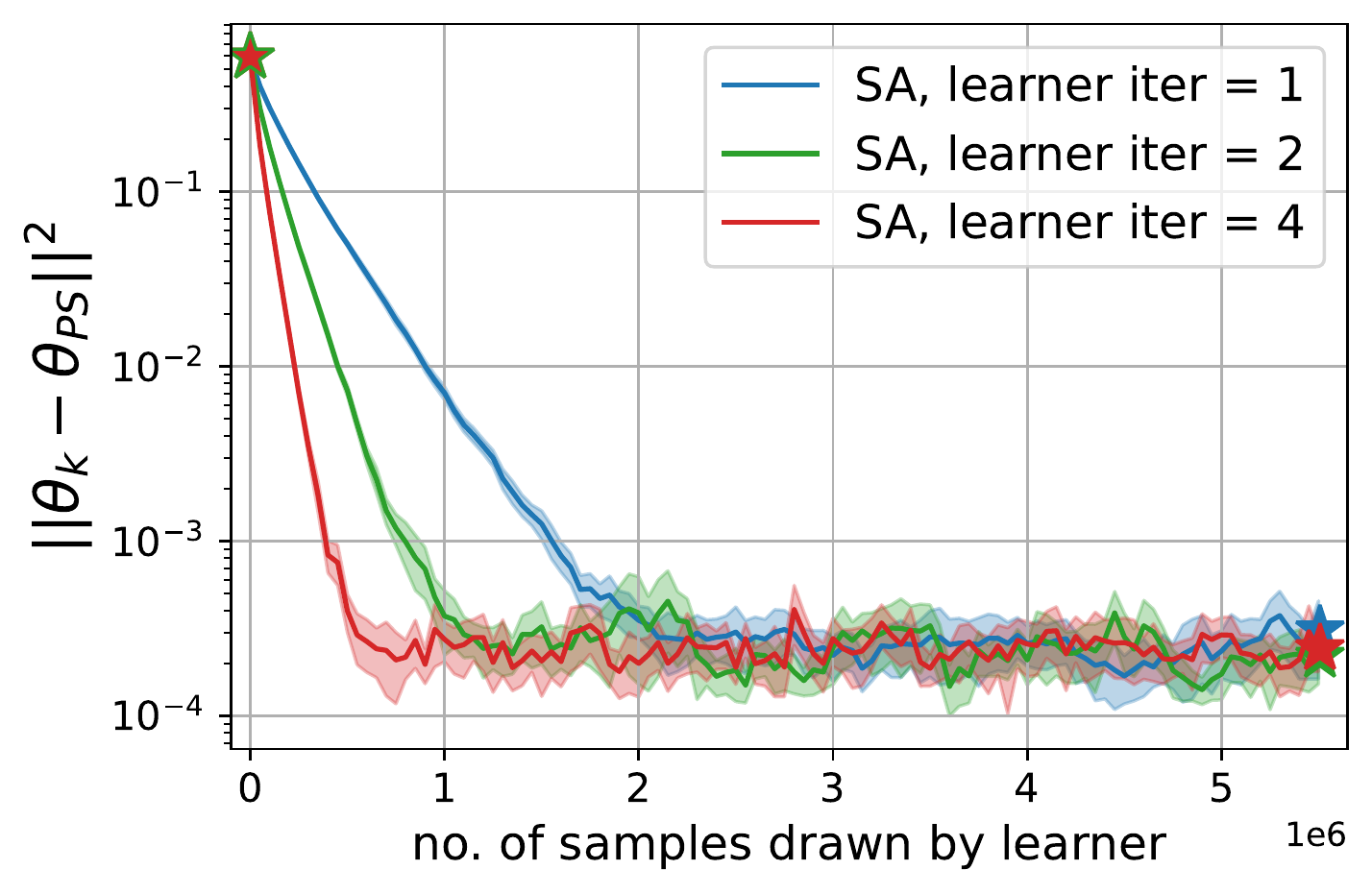}~\includegraphics[width=.45\linewidth]{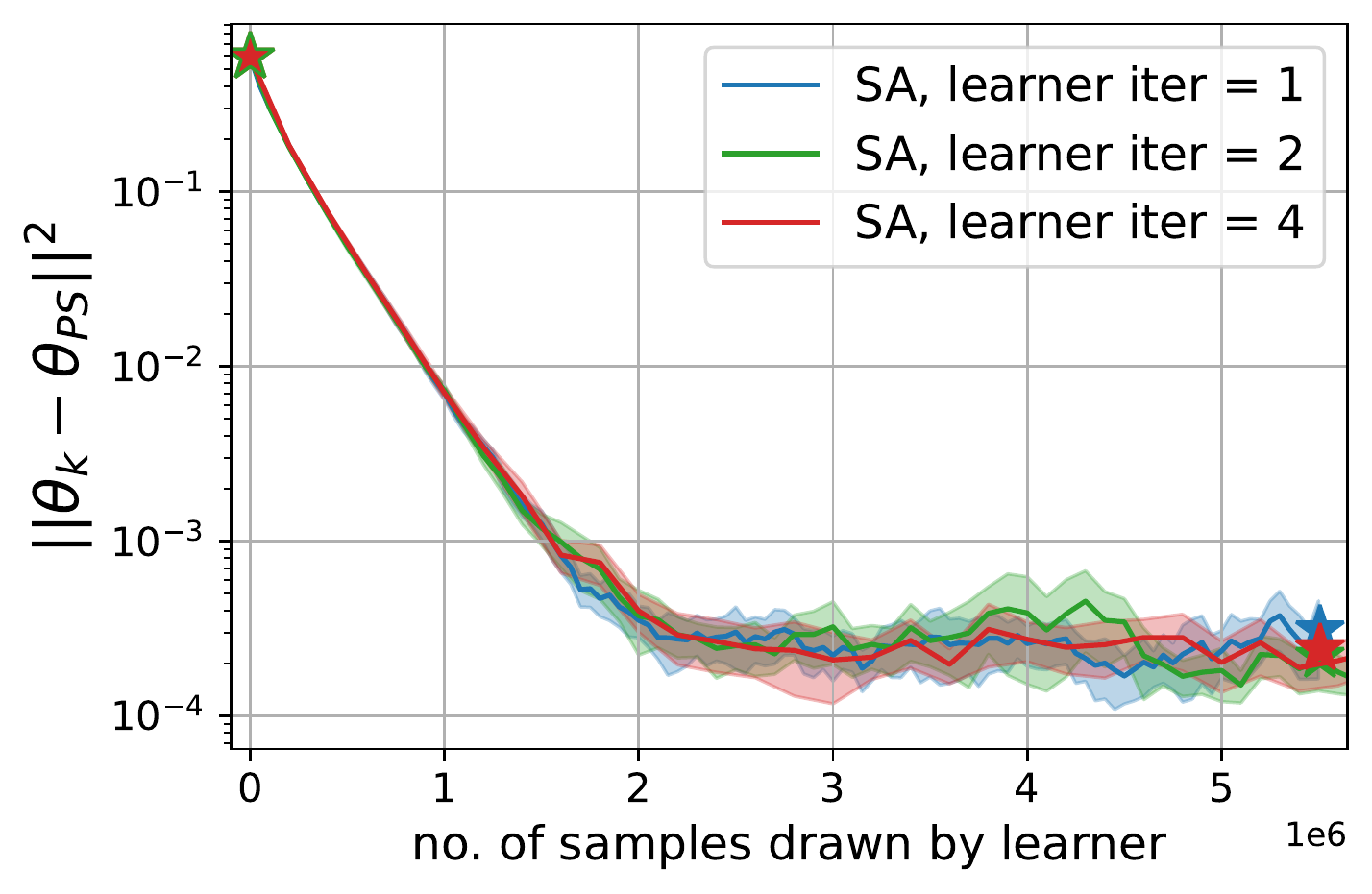}
     \caption{Convergence of SA algorithm with varying number of learner's updates per iteration.}
     \label{fig:lazy comparision}\vspace{-.2cm}
 \end{figure}
 
\paragraph{Effects of Stateful Updates at Agents} Notice that the comparison has been made in Fig.~\ref{fig:main_sim} (right). Here, we again plot the convergence of the SA algorithm 
to illustrate the convergence rates from the learner's perspective as well. We observe that the SA algorithms with stateful update converges as $k$ increases.
We vary the `learner's iteration' parameter to observe the effects on convergence when the learner is adapting at faster rate than the agents. This is achieved by repeating steps 2 and 3 in Algorithm 2 for multiple times. Notice that this setting is similar to the lazy deploy scheme in \citep{mendler2020stochastic}. From the figure, we observe that doing so improves the convergence from the agents' perspective, while the sample efficiency (from the learner's perspective) is unaffected. 

 \begin{figure}[h]
     \centering
     \includegraphics[width=.45\linewidth]{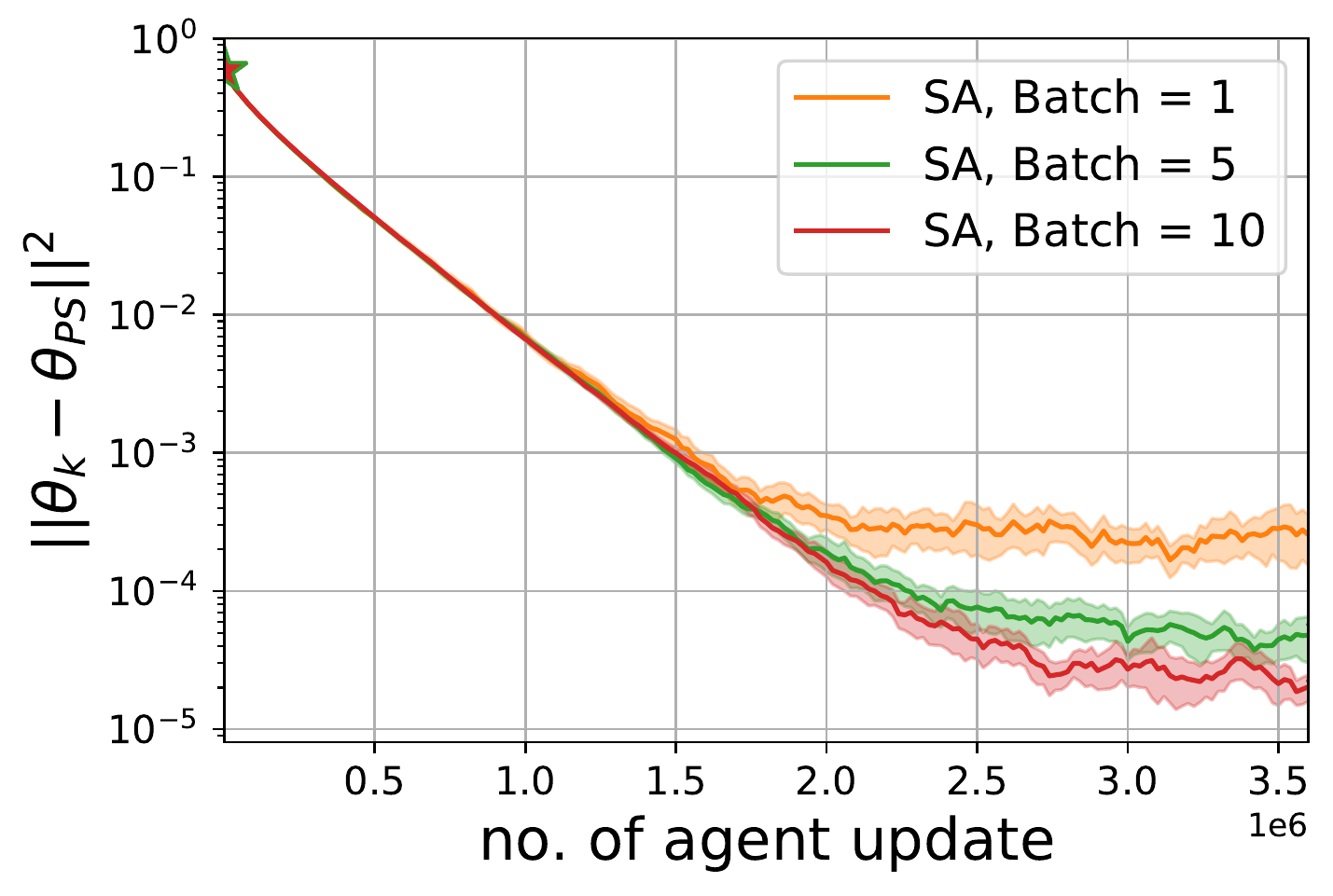}~\includegraphics[width=.45\linewidth]{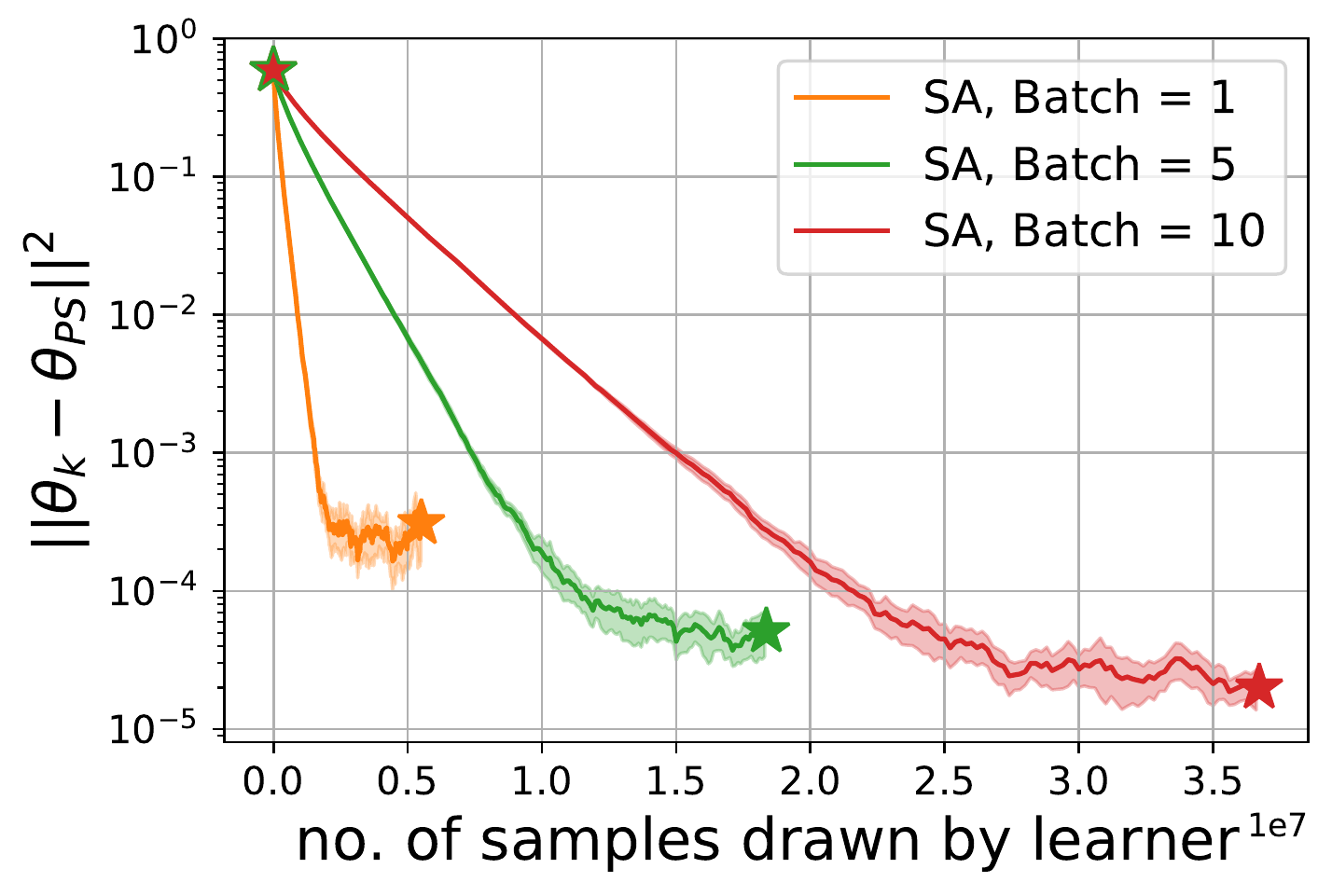}
     \caption{Convergence of SA algorithm with varying number of minibatch size.}\vspace{-.2cm}
     \label{fig:batch size comparision}
 \end{figure}
 
\paragraph{Effects of Minibatch Size} In this experiment, we consider the setting of $U_{\sf lg}(\cdot)$ and we draw different batch size of samples ($\hat{\sf b} \in \{1,5,10\}$) per iteration. To implement this, at step 2 of Algorithm 2, the learner draws $\hat{\sf b}$ agents uniformly as $\hat{\cal I}_k$, and at step 3, we update the iterate through:
\[
\theta_{k+1} = \theta_k - \gamma_{k+1} \frac{1}{\hat{\sf b}} \sum_{ j \in \hat{\cal I}_k } \grd \ell( \theta_k ; z_{k+1, j}).
\]
In Fig.~\ref{fig:batch size comparision}, we compare the error $\| \theta_k - \theta_{PS} \|^2$ in terms of the number of agents' best response update\footnote{Since the agents only perform one best response update per iteration, the $x$-axis here is equivalent to the iteration number $k$.} and the number of samples drawn by the \emph{learner}. 
We find that increasing the minibatch reduces the variance of the gradient estimate, yet it can be less sample efficient from the learner's perspective. 

  \begin{figure}[h]
     \centering
     \includegraphics[width=.45\linewidth]{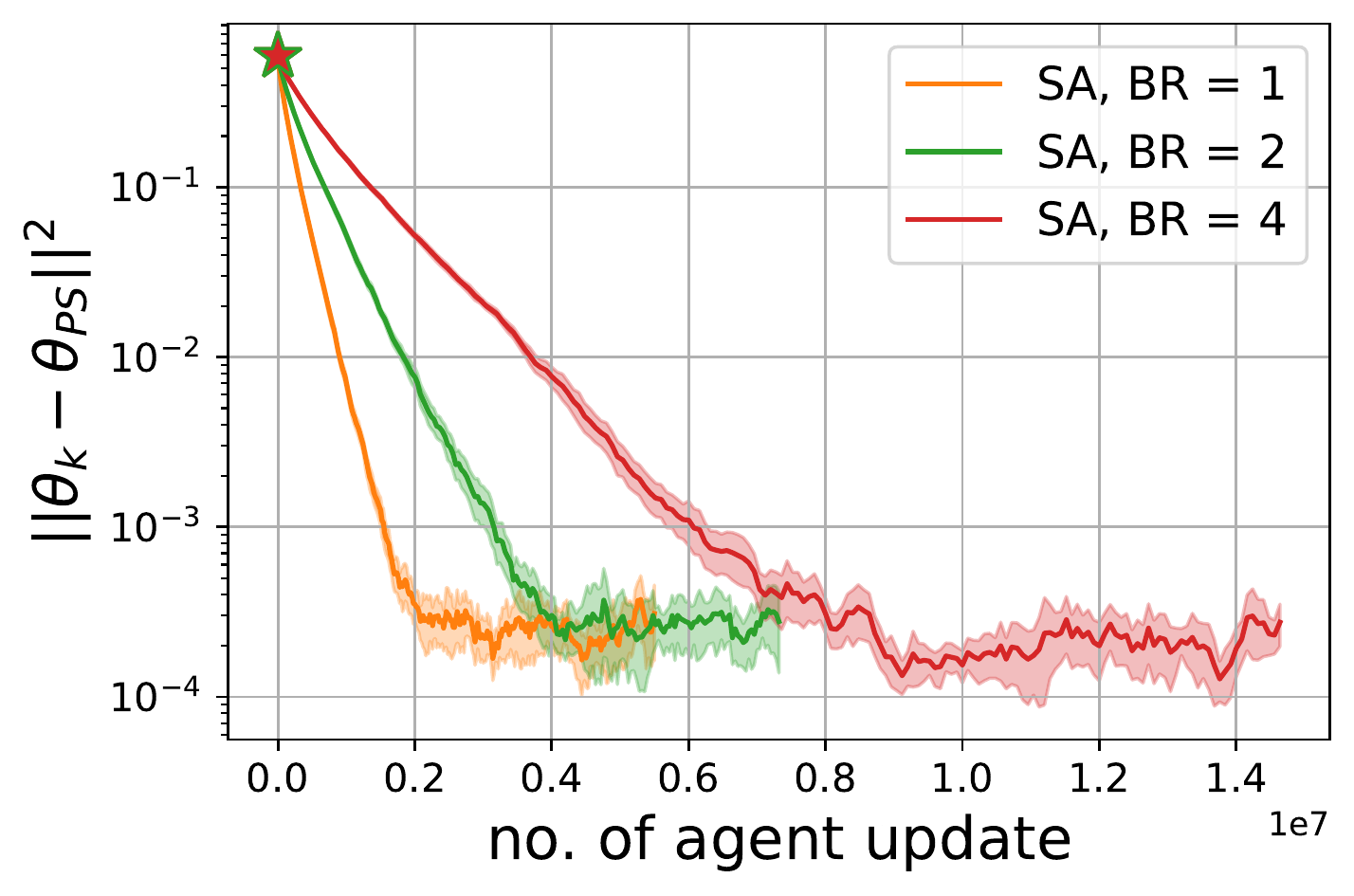}~\includegraphics[width=.45\linewidth]{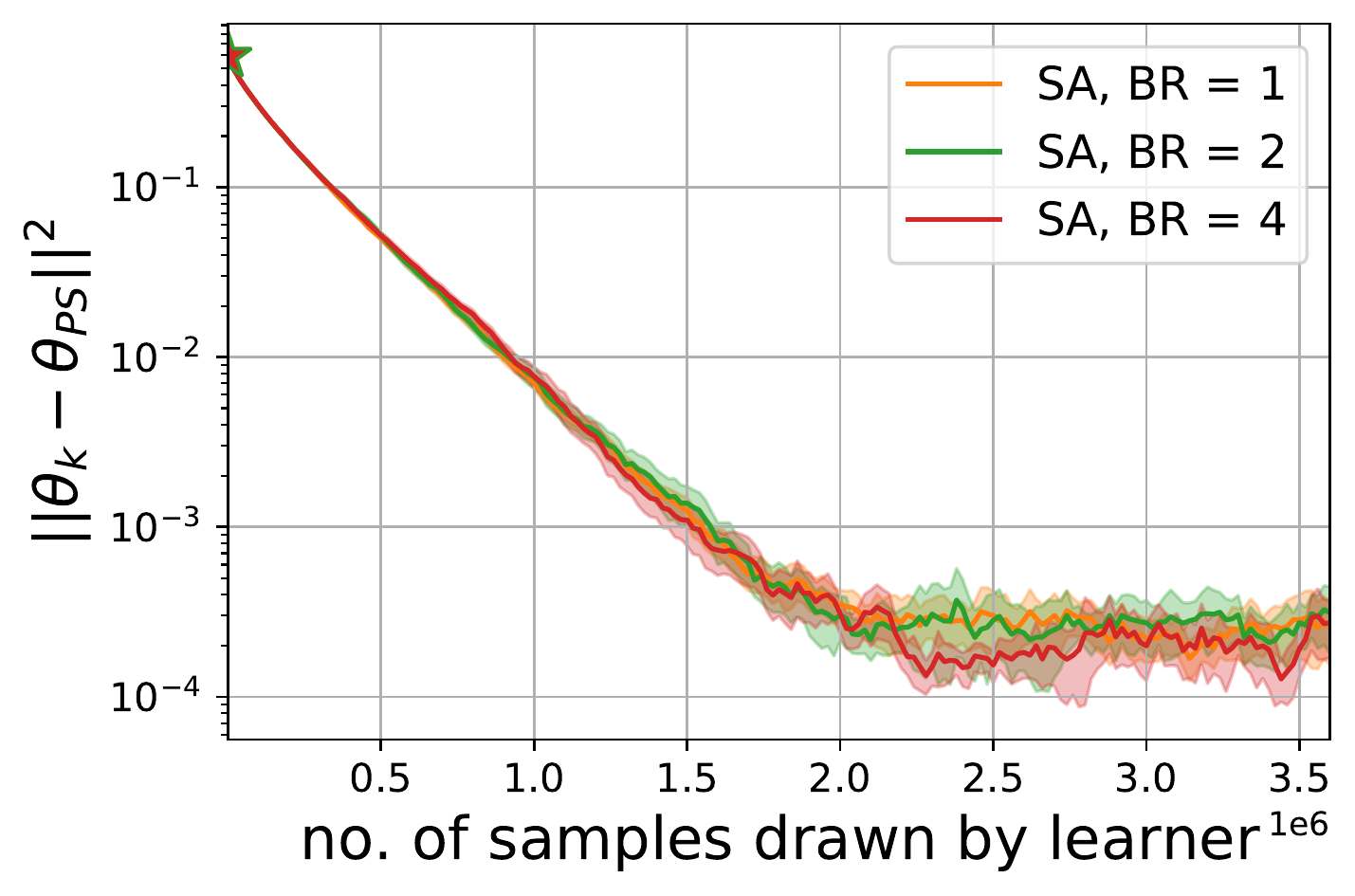}
     \caption{Convergence of SA algorithm with varying number of best responses.}\vspace{-.2cm}
     \label{fig:br comparison}
 \end{figure}
 
\paragraph{Effects of Number of Adaptive Best Responses} In this experiment, we consider the setting of $U_{\sf lg}(\cdot)$ and at each iteration, the agents execute multiple rounds of adapted best response (${\sf BR} \in \{1,2,4\}$) to simulate the scenario when the agents are allowed with more time to respond to the published classifier $\theta_k$. To implement this, we repeat the update in \eqref{eq:abr} of step 1 in Algorithm~2 for ${\sf BR}$ times.
Notice that this is reverse of Fig.~\ref{fig:main_sim} (right) where the learner performs multiple iterations per agents' best response update. 

In Fig.~\ref{fig:br comparison}, we compare the error $\| \theta_k - \theta_{PS} \|^2$ in terms of the number of agent update and the number of samples drawn by the \emph{learner}. We observe that increasing the number of best responses improves the performance slightly. However, as a drawback, it requires more computations/updates at the agents to reach the same performance.

\end{document}